\documentclass[a4paper,reqno,12pt]{amsart}

\usepackage{amsmath}
\usepackage{amsthm}
\usepackage{amssymb}
\usepackage{latexsym}
\usepackage{graphicx,color}
\usepackage{booktabs}

\numberwithin{equation}{section}

\newtheorem{thm}{Theorem}[section]
\newtheorem*{thm*}{Theorem}
\newtheorem*{mainthm*}{Main Theorem}
\newtheorem{prop}[thm]{Proposition}
\newtheorem{lem}[thm]{Lemma}
\newtheorem{cor}[thm]{Corollary}

\theoremstyle{definition}

\theoremstyle{remark}
\newtheorem{rem}{Remark}

\allowdisplaybreaks[3]

\newcommand{\Ga}{\Gamma}
\newcommand{\de}{\delta}
\newcommand{\De}{\Delta}
\newcommand{\e}{\varepsilon}

\newcommand{\la}{\lambda}
\newcommand{\La}{\Lambda}
\newcommand{\sgm}{\sigma}
\renewcommand{\th}{\theta}
\newcommand{\p}{\partial}
\newcommand{\I}{\infty}
\newcommand{\Sc}[1]{\mathcal{#1}}
\newcommand{\F}{\Sc{F}}

\newcommand{\Bo}[1]{\mathbb{#1}}
\newcommand{\R}{\Bo{R}}
\newcommand{\T}{\Bo{T}}
\newcommand{\Tl}{\Bo{T}_\la}
\newcommand{\Zl}{\Bo{Z}_\la}
\newcommand{\lec}{\lesssim}
\newcommand{\gec}{\gtrsim}

\newcommand{\bbar}{\overline}
\newcommand{\ti}{\widetilde}

\newcommand{\supp}[1]{\> \operatorname{supp}\> #1}

\newcommand{\shugo}[1]{\{ #1\}}
\newcommand{\Shugo}[2]{\big\{ \, #1 \, \big| \, #2 \, \big\}}
\newcommand{\LR}[1]{{\langle #1 \rangle }}
\newcommand{\chf}[1]{\textbf{1}_{#1}}
\newcommand{\proj}[1]{P_{\shugo{#1}}}
\newcommand{\eq}[2]{\begin{equation} \label{#1} \begin{split} #2 \end{split} \end{equation}}
\newcommand{\eqq}[1]{\begin{equation*} \begin{split} #1 \end{split} \end{equation*}}
\newcommand{\mat}[1]{\begin{smallmatrix} #1 \end{smallmatrix}}
\newcommand{\norm}[2]{\big\| #1 \big\| _{#2}}
\newcommand{\tnorm}[2]{\| #1 \| _{#2}}
\newcommand{\hx}{\hspace{10pt}}

\newcommand{\ttfrac}[2]{\text{{\footnotesize $\frac{#1}{#2}$}}}
\newcommand{\ttsum}{\textstyle\sum\limits}
\newcommand{\ttsumla}{\ttfrac{1}{\la} \textstyle\sum\limits}

\newcommand{\midashi}[1]{\noindent \textbf{\underline{~#1~}}}

\newcommand{\eqs}[1]{\begin{gather*} #1 \end{gather*}}

\title[``Good'' Boussinesq equation on torus]{Sharp local well-posedness for the ``good'' Boussinesq equation}
\author[Nobu Kishimoto]{Nobu Kishimoto}
\address{Department of Mathematics, Kyoto University, Kyoto 606-8502, Japan}
\email{n-kishi@math.kyoto-u.ac.jp}

\begin{document}

\begin{abstract}
In the present article, we prove the sharp local well-posedness and ill-posedness results for the ``good'' Boussinesq equation on $\Bo{T}$; the initial value problem is locally well-posed in $H^{-1/2}(\Bo{T})$ and ill-posed in $H^s(\Bo{T})$ for $s<-\frac{1}{2}$.
Well-posedness result is obtained from reduction of the problem into a quadratic nonlinear Schr\"odinger equation and the contraction argument in suitably modified $X^{s,b}$ spaces.
The proof of the crucial bilinear estimates in these spaces, especially in the lowest regularity, rely on some bilinear estimates for one dimensional periodic functions in $X^{s,b}$ spaces, which are generalization of the bilinear refinement of the $L^4$ Strichartz estimate on $\R$.
Our result improves the known local well-posedness in $H^s(\Bo{T})$ with $s>-\frac{3}{8}$ given by Oh and Stefanov (2012) to the regularity threshold $H^{-1/2}(\Bo{T})$.
Similar ideas also establish the sharp local well-posedness in $H^{-1/2}(\R )$ and ill-posedness below $H^{-1/2}$ for the nonperiodic case, which improves the result of Tsugawa and the author (2010) in $H^s(\R )$ with $s>-\frac{1}{2}$ to the limiting regularity.
\end{abstract}

\maketitle

\section{Introduction}

We investigate the following initial value problem for the ``good'' Boussinesq equation (GB):
\begin{equation}\label{gB}
\left\{
\begin{array}{@{\,}r@{\;}l}
&\p _t^2v-\p _x^2v+\p _x^4v+\p _x^2(v^2)=0,\qquad (t,x) \in [-T,T] \times Z,\\
&v(0,x)=v_0(x),\quad \p _tv(0,x)=v_1(x),
\end{array}
\right.
\end{equation}
where $Z=\R$ or $\Bo{T}:=\R/2\pi \Bo{Z}$.
The unknown function may be real-valued or complex-valued.
The principal aim of this article is to establish the sharp well-posedness and ill-posedness results for \eqref{gB} in Sobolev spaces.

In the 1870's, Boussinesq proposed some model equations for the propagation of shallow water waves as the first mathematical model for the phenomenon of solitary waves which had been observed by Scott-Russell in the 1840's.
One of his equations may be written in the form
\eq{bB}{\p _t^2v-\p _x^2v-\p _x^4v+\p _x^2(v^2)=0,}
which we call ``bad'' Boussinesq equation in contrast with \eqref{gB}.
In fact, \eqref{bB} is linearly unstable due to the exponentially growing Fourier components, though it has a Lax pair formulation and admits the inverse scattering approach (\cite{Z74,DTT82}).
These equations arise as a model for the nonlinear strings (\cite{Z74}), while the Boussinesq type equations of ``good'' sign also arise in the study of shape-memory alloys (\cite{FLS87}).
It is also known that solutions to GB may blow up in finite time (\cite{KL78,S90}).

Let us review some of the known results on the local well-posedness (LWP) of \eqref{gB} in Sobolev spaces.
The first result may go back to Bona and Sachs \cite{BS88}, who applied Kato's theory of quasilinear evolution equations to establish LWP for initial data $(v_0,v_1)$ in, roughly speaking, $H^s(\R)\times H^{s-2}(\R)$ with $s>\frac{5}{2}$.
They also showed the nonlinear stability of solitary wave solutions to \eqref{gB} which leads to the global existence of solutions close to a solitary wave.

Note that \eqref{gB} is formally rewritten as 
\eq{IEgB}{v(t)=&\cos (t\sqrt{-\p _x^2+\p _x^4})v_0+\frac{\sin (t\sqrt{-\p _x^2+\p _x^4})}{\sqrt{-\p _x^2+\p _x^4}}v_1\\
&\hx +\int _0^t \sin ((t-t')\sqrt{-\p _x^2+\p _x^4})\frac{-\p _x^2}{\sqrt{-\p _x^2+\p _x^4}}v^2(t')\,dt',}
in which the loss of two derivatives in the nonlinearity is totally recovered.
Thus, one expects that the Strichartz type inequalities are effective for lower regularities.
Linares \cite{L93} exactly did that and showed LWP of \eqref{gB} for roughly $(v_0,v_1)\in L^2(\R )\times H^{-2}(\R)$.
We see that the difference of regularities between $v_0$ and $v_1$ is natural from the viewpoint of the integral formulation \eqref{IEgB}. 

Now, we recall the relation between \eqref{gB} and quadratic nonlinear Schr\"odinger equations mentioned in \cite{KT10}.
Consider the real-valued case for simplicity (for the complex-valued case we refer to \cite{KT10}).
Putting $u:=v+i(1-\p _x^2)^{-1}\p _tv$, the Cauchy problem \eqref{gB} transforms into
\begin{equation}\label{gB'}
\left\{
\begin{array}{@{\,}r@{\;}l}
&i\p _tu+\p _x^2u=\frac{1}{2}(u-\bar{u})-\frac{1}{4}\omega ^2(u+\bar{u})^2,\qquad (t,x) \in [-T,T] \times Z,\\
&u(0,x)=u_0(x),
\end{array}
\right.
\end{equation}
where the unknown function is complex-valued and
\eqq{\omega ^2:=\frac{-\p _x^2}{1-\p _x^2},\qquad u_0=v_0+i(1-\p _x^2)^{-1}v_1.}
We can recover \eqref{gB} from \eqref{gB'} by putting $v:=\Re u$, $(v_0,v_1):=(\Re u_0, (1-\p _x^2)\Im u_0)$.
The mappings
\eqs{H^s(\Bo{T};\R )\times H^{s-2}(\Bo{T};\R )\ni (v_0,v_1)\mapsto u_0\in H^s(\Bo{T};\Bo{C}),\\
C([0,T];H^s(\Bo{T};\R ))\cap C^1([0,T];H^{s-2}(\Bo{T};\R ))\ni v\mapsto u\in C([0,T];H^s(\Bo{T};\Bo{C}))}
are bi-Lipschitz, so LWP of \eqref{gB} in $H^s\times H^{s-2}$ is equivalent to that of \eqref{gB'} in $H^s$.
Since the term $\frac{1}{2}(u-\bar{u})$ in \eqref{gB'} is harmless, the ``good'' Boussinesq equation can be regarded essentially as the nonlinear Schr\"odinger equation with nonlinearity $\omega ^2(F_1(u,u)+F_2(u,u)+F_3(u,u))$, where $F_1(u,v):=uv$, $F_2(u,v)=u\bar{v}$, and $F_3(u,v)=\bar{u}\bar{v}$.

The initial value problem of quadratic nonlinear Schr\"odinger equations (qNLS)
\begin{equation}\label{qNLS}
\left\{
\begin{array}{@{\,}r@{\;}l}
&i\p _tu+\p _x^2u=F_j(u,u),\qquad (t,x) \in [-T,T] \times Z,\\
&u(0,x)=u_0(x)\in H^s(Z),\qquad j=1,2,3
\end{array}
\right.
\end{equation}
has been extensively studied since Bourgain \cite{B93-1} introduced the $X^{s,b}$ norms (see \eqref{def_Xsb} for the definition).
When we apply the $X^{s,b}$ norm method (\mbox{i.e.} Picard iteration method using the $X^{s,b}$ norms), LWP of \eqref{qNLS} in $H^s$ is often reduced to some bilinear estimate in $X^{s,b}$, typically as follows:
\eqq{\norm{F_j(u,v)}{X^{s,b-1}}\lec \norm{u}{X^{s,b}}\norm{v}{X^{s,b}}.}
(We usually take $b$ such that $\frac{1}{2}<b\,(<1)$ to keep the $X^{s,b}$ norm stronger than $L^\I_tH^s_x$.)
Also for GB, the $X^{s,b}$ norm method has provided substantial progress in low regularity theory. 
See Table~1 for the best known results on the local well-posedness of \eqref{gB} and \eqref{qNLS}.
The result of Fang and Grillakis \cite{FG96} applied the argument of Bourgain to the case of \eqref{gB} on torus and proved LWP in $L^2$, just the same as the best regularity obtained for \eqref{qNLS} with nonlinearity $F_2$.
(Results on GB should be compared with the worst results on qNLS among $F_j$'s, since the nonlinearity in \eqref{gB'} includes all of them.)
Similarly, Farah \cite{F09} and Farah, Scialom \cite{FS10} successfully adapted the argument of Kenig, Ponce, Vega \cite{KPV96-NLS} for qNLS in the $\R$ case and the $\Bo{T}$ case, respectively, obtaining LWP in $H^s\times H^{s-2}$ with $s>-\frac{1}{4}$ in both cases.
These results in \cite{KPV96-NLS,F09,FS10} are the best that one can show by the standard iteration argument in $X^{s,b}$, because the bilinear estimates in $X^{s,b}$ (with $b>\frac{1}{2}$) fail if $s$ is lower than these thresholds.

\begin{table}
\begin{center}
\begin{tabular}{|c|c|c|c|}
\hline 
\makebox(1,0)[t]{$Z$}&\multicolumn{2}{c|}{Quadratic Schr\"odinger equations \eqref{qNLS}} &\makebox(170,0)[t]{``Good'' Boussinesq equation \eqref{gB}}\\
\cline{2-3}
&~$F_1(u,u)$, $F_3(u,u)$~&$F_2(u,u)$&\\
\hline
\multicolumn{4}{|c|}{Bilinear estimate in $X^{s,b}$ ($b>\frac{1}{2}$)}\\
\hline
$\R$&$s>-\frac{3}{4}$ \cite{KPV96-NLS}&$s>-\frac{1}{4}$ \cite{KPV96-NLS}&$s>-\frac{1}{4}$ \cite{F09}\\
\hline
$\Bo{T}$&$s>-\frac{1}{2}$ \cite{KPV96-NLS}&$s\ge 0$ \cite{B93-1}&$s>-\frac{1}{4}$ \cite{FS10}\\
\hline
\multicolumn{4}{|c|}{Local well-posedness in $H^s$ (in $H^s\times H^{s-2}$ for \eqref{gB})}\\
\hline
$\R$&$s\ge -1$ $^*$ \cite{BT06,K09-NLS}&$s\ge -\frac{1}{4}$ $^*$ \cite{KT10}&$s>-\frac{1}{2}$ \cite{KT10}\\
\hline
$\Bo{T}$&$s>-\frac{1}{2}$ \cite{KPV96-NLS}&$s\ge 0$ $^*$ \cite{B93-1}&$s>-\frac{3}{8}$ \cite{OS12}\\
\hline 
\end{tabular}
\end{center}
\caption{The best known results on the local well-posedness of \eqref{gB} and \eqref{qNLS}.
$^*$ indicates the optimality of these results in the sense that the data-to-solution map fails to be continuous below these regularity thresholds.}
\end{table}

It is worth noting the difference between \eqref{gB} and \eqref{qNLS}: Concerning the bilinear estimate, the required regularity for $\Bo{T}$ is $\frac{1}{4}$ worse than that for $\R$ in the results on qNLS, while there is no difference in the results on GB.
To explain this, we should note that the worst nonlinear interaction which breaks the bilinear estimate in $X^{s,b}$ is of \emph{high $\times$ high $\to$ low} type, \mbox{i.e.}, the interaction of two components in high frequency $\shugo{|\xi |\gg 1}$ brings component in low frequency $\shugo{|\xi |<1}$.
The contribution from low frequency is severer on torus than on $\R$, which explains the difference of required regularities between $\R$ and $\Bo{T}$ in the case of qNLS.
However, the additional operator $\omega ^2$ in GB acts as $\p _x^2$ in low frequency and reduces significantly (completely, in the torus case) the low frequency component.
That is why the difference becomes less clear in GB. 

Lack of the bilinear estimate in $X^{s,b}$, however, does not necessarily imply ill-posedness of the problem.
For instance, there may be a chance that one can recover the bilinear estimate by changing function spaces.
In fact, Bejenaru and Tao~\cite{BT06} introduced a suitably modified $X^{s,b}$ space for the problem \eqref{qNLS} on $\R$ with nonlinearity $F_1$, which captures the worst nonlinear interaction in this case and restores the bilinear estimate, extending the previous result in \cite{KPV96-NLS} from $s>-\frac{3}{4}$ to $s\ge -1$.
They also provided a general machinery to show ill-posedness, and actually obtained the ill-posedness of this problem for $s<-1$.
Their ideas were refined further by the author \cite{K09-NLS} to give the same conclusion for the case of another nonlinearity $F_3$, and also appeared in the work of Tsugawa and the author \cite{KT10} treating \eqref{qNLS} with nonlinearity $F_2$.
The idea of modifying $X^{s,b}$ is also effective for GB, and the previous LWP results for \eqref{gB} on $\R$ was improved in \cite{KT10} to $s>-\frac{1}{2}$.
On the other hand, a different approach was recently taken by Oh and Stefanov~\cite{OS12} to push down the regularity threshold for GB on $\T$ to $s>-\frac{3}{8}$; they applied the method of normal forms to show that the Duhamel part of the nonlinear solution is much smoother than the free solution.

In this article, following \cite{KT10}, we shall perform more refined modification of the $X^{s,b}$ norms and establish the sharp LWP results for GB on both $\R$ and $\Bo{T}$.
The main result is as follows.
\begin{thm}\label{thm_main}
Let $Z$ be $\R$ or $\Bo{T}$.
Then the initial value problem \eqref{gB'} in $H^s(Z)$ is locally well-posed for $s\ge -\frac{1}{2}$ and ill-posed for $s<-\frac{1}{2}$.
More precisely, we have the following.

(I) Let $-\frac{1}{4}\ge s\ge -\frac{1}{2}$.
Then, for any $r>0$ and any $u_0\in H^s$ with $\tnorm{u_0}{H^s}\le r$, there exists a solution $u\in C([-T,T];H^s)$ to the integral equation associated to \eqref{gB'} with the existence time $T=T(r)>0$.
Moreover, the solution is uniquely obtained in some Banach space $W^s_T$ embedded continuously into $C([-T,T];H^s)$, and the data-to-solution map from $\Shugo{u_0\in H^s}{\tnorm{u_0}{H^s}\le r}$ to $W^s_T$ is Lipschitz.

(II) Let $s<-\frac{1}{2}$.
Then, there exists $T_0>0$ such that for any $0<t_0\le T_0$ the flow map of \eqref{gB'}, $u_0\in H^{-1/2}\mapsto u(t_0)\in H^{-1/2}$ (defined in (I) for sufficiently small data), is not continuous at the origin as a map on $H^s$.
\end{thm}

Concerning (I), the definition of the function space $W^s_T$ (modification of $X^{s,b}$) for the periodic case with $s>-\frac{1}{2}$ will be essentially the same as that for $\R$ given in \cite{KT10}.
Proof of the key bilinear estimate, which will be given in Section~\ref{sec_be+}, is also simple, based on some well-known estimates such as Bourgain's $L^4$ Strichartz estimate~\cite{B93-1} and the Sobolev embeddings.
The limiting case $s=-\frac{1}{2}$ is much more difficult to deal with, and we will have to refine further the definition of the function space and exploit some estimates including gain of derivatives.

For (II), we follow the argument in \cite{KT10} which showed the ill-posedness for qNLS of $F_2$ type below $H^{-1/4}$.
This kind of argument was previously established by Bejenaru and Tao \cite{BT06} in more abstract settings, as mentioned above.
They showed that one can upgrade discontinuity of one of the Picard iterates to discontinuity of the whole nonlinear solution map, \emph{in some special situations}.
It should be emphasized that, in such situations, the LWP estimates for the limiting regularity ($s=-\frac{1}{2}$ in our case) should be required for ill-posedness below that regularity.
In fact, we will show the discontinuity (unboundedness) of the second iterate for the problem and apply the argument mentioned above, but it is not possible without LWP for $s=-\frac{1}{2}$.

As a corollary of Theorem~\ref{thm_main}, we establish the sharp LWP and ill-posedness of \eqref{gB}.
\begin{cor}
Let $Z$ be $\R$ or $\Bo{T}$.
Then the initial value problem \eqref{gB} in $H^s(Z)\times H^{s-2}(Z)$ (real- or complex-valued) is locally well-posed for $s\ge -\frac{1}{2}$ and ill-posed for $s<-\frac{1}{2}$.
\end{cor}

The paper is organized as follows.
In the next section we will define the function spaces and show the required estimates except for the bilinear estimate.
Proof of Theorem~\ref{thm_main}~(I) will be also given.
In Section~\ref{sec_L4}, we will prepare some modified version of the $L^4$ Strichartz estimate for periodic functions.
We will give a proof of the crucial bilinear estimate, for the case $s>-\frac{1}{2}$ in Section~\ref{sec_be+} and for the limiting case $s=-\frac{1}{2}$ in Section~\ref{sec_be0}.
Finally, a proof of Theorem~\ref{thm_main}~(II) will be given in Section~\ref{sec_illp}.
The Appendix section will be devoted to the proof of some estimate which we will use to derive the uniqueness of solutions.


\bigskip
\section{Preliminaries}\label{sec_def}

We begin with the scaling argument.
When $u(t,x)$ solves \eqref{gB'}, 
\eqq{u^\la (t,x):=\la ^{-2}u(\la ^{-2}t,\la ^{-1}x),\qquad \la >0}
solves the following rescaled initial value problem:
\begin{equation}\label{gB''}
\left\{
\begin{array}{@{\,}r@{\;}l}
&i\p _tu^\la +\p _x^2u^\la =\frac{1}{2}\la ^{-2}(u^\la -\bbar{u^\la })-\frac{1}{4}\omega _\la ^2(u^\la +\bbar{u^\la })^2,\quad (t,x) \in [-\la ^2T,\la ^2T] \times Z_\la,\\
&u^\la (0,x)=u_0^\la (x)\in H^s(\Bo{T}_\la)
\end{array}
\right.
\end{equation}
with $u_0^\la (x):=\la ^{-2}u_0(\la ^{-1}x)$, where $Z_\la =\R$ if $Z=\R$ and $Z_\la =\Bo{T}_\la :=\R /(2\pi \la \Bo{Z})$ if $Z=\Bo{T}$.
We have also used the notation
\eqs{\omega _\la ^2=\F _\xi ^{-1}\ttfrac{\la ^2\xi ^2}{1+\la ^2\xi ^2}\F _x.}
For the torus case, we define the Fourier coefficient of a $2\pi \la$ periodic function $\phi$ in the usual fashion as
\eqs{\F_x\phi (k):=\ttfrac{1}{\sqrt{2\pi}}\int _0^{2\pi \la}e^{-ikx}\phi (x)\,dx,\quad k\in \Zl :=\Bo{Z}/\la .}
$\Zl$ is equipped with the normalized counting measure and for $f:\Bo{Z}_\la \to \Bo{C}$,
\eqs{\norm{f}{\ell ^p(\Zl )}:=\big( \ttfrac{1}{\la}\textstyle\sum\limits _{k\in \Zl}|f(k)|^p\big) ^{1/p},\\
\norm{\phi}{H^s(\Tl )}:=\norm{\LR{k}^s\F _x\phi (k)}{\ell ^2_k(\Zl )},\qquad \LR{\,\cdot\,}:=(1+|\,\cdot\,|^2)^{1/2}.}


A simple calculation shows that if $s<0$, we have
\eqq{\norm{u_0^\la}{H^s(Z_\la )}\le \la ^{-s-3/2}\norm{u_0}{H^s(Z_\la )}}
for $\la \ge 1$.
In the following, we treat $0>s>-\frac{3}{2}$ and construct solutions to the rescaled problem \eqref{gB''} with $\la \ge 1$ on the time interval $[-1,1]$ for initial data sufficiently small in $H^s(Z_\la )$.

The $X^{s,b}$ spaces for spacetime functions $u(t,x)$ on $\R \times Z_\la$ is defined via the following $X^{s,b}$ norm
\eq{def_Xsb}{\norm{u}{X^{s,b}(\R \times Z_\la )}:=\norm{\LR{\xi}^s\LR{\tau +\xi ^2}^b\ti{u}}{L^2_\xi (Z_\la^* ;L^2_{\tau}(\R ))},}
where $\ti{u}=\F _{t,x}u$ denotes the spacetime Fourier transform of $u$ and $Z^*_\la=\R$ or $\Bo{Z}_\la$.
Also define the $Y^s$ spaces by
\eqq{\norm{u}{Y^s}:=\norm{\LR{\xi}^s\ti{u}}{L^2_\xi L^1_\tau}.}

Now, we define the space $W^s$, which is modification of $X^{s,b}$, by the following norm
\eqq{\norm{u}{W^s}&:=\norm{P_{\shugo{\LR{\tau +\xi ^2}\lec \LR{\xi}}}u}{X^{s,1}}\\
&\hx +\norm{P_{\shugo{\LR{\tau +\xi ^2}\gec \LR{\xi}}}u}{X^{s+1,0}}+\norm{P_{\shugo{\LR{\tau +\xi ^2}\gg \LR{\xi}^2}}u}{Y^s}\qquad \text{for}~-\ttfrac{1}{4}\ge s>-\ttfrac{1}{2},\\
\norm{u}{W^{-1/2}}&:=\norm{P_{\shugo{\LR{\tau +\xi ^2}\lec \LR{\xi}}}u}{X^{-1/2,1}}\\
&\hx +\norm{P_{\shugo{\LR{\xi}\lec \LR{\tau +\xi ^2}\lec \LR{\xi}^2}}u}{X^{1/2,0}}\\
&\hx +\ttsum _{M\ge 1\;;\;\text{dyadic}}\norm{P_{\shugo{\LR{\tau +\xi ^2}\sim M}\cap \shugo{\LR{\tau +\xi ^2}\gec \LR{\xi}^2}}u}{X^{1/2,0}}+\norm{P_{\shugo{\LR{\tau +\xi ^2}\gg \LR{\xi }^2}}u}{Y^{-1/2}},}
where we denote by $P_\Omega$ the spacetime Fourier projection onto a set $\Omega \subset \R \times Z^*_\la$.

\begin{rem}
In \cite{KT10} we have used similar spaces defined by
\eqq{\norm{u}{Z^s}&:=\norm{P_{\shugo{\LR{\tau +\xi ^2}\lec \LR{\xi}}}u}{X^{s,1}}\\
&\hx +\norm{P_{\shugo{\LR{\tau +\xi ^2}\gec \LR{\xi}}}u}{X^{1/2,s+1/2}}+\norm{P_{\shugo{\LR{\tau +\xi ^2}\gg \LR{\xi}^2}}u}{Y^s},\qquad -\ttfrac{1}{4}>s>-\ttfrac{1}{2}}
for the nonperiodic case.
\end{rem}

For $T>0$, define the restricted space $W^s_T$ by the restrictions of distributions in $W^s$ to $(-T,T)\times Z_\la$, with the norm
\eqq{\norm{u}{W^s_T}:=\inf \Shugo{\norm{U}{W^s}}{U\in W^s~\text{is an extension of $u$ to $\R \times Z_\la$}}.}
This notation will be used for various function spaces of spacetime functions.

These spaces obey the following embeddings.
\begin{lem}\label{lem_embedding}
For $-\frac{1}{4}\ge s\ge -\frac{1}{2}$ and $0< \th \le 1$ with $(s,\th )\neq (-\frac{1}{2},1)$, we have
\eqq{X^{s,1},\,X^{s+\th ,1-\th}\cap Y^s\hookrightarrow W^s\hookrightarrow X^{s,0}\cap Y^s.}
\end{lem}
\begin{proof}
$W^s\hookrightarrow X^{s,0}$ is trivial by the definition.
For $W^s\hookrightarrow Y^s$, it suffices to show
\eqq{\norm{P_{\shugo{\LR{\tau +\xi ^2}\lec N^2}}u_N}{Y^s}\lec \norm{P_{\shugo{N\lec \LR{\tau +\xi ^2}\lec N^2}}u_N}{X^{s+1,0}}+\norm{P_{\shugo{\LR{\tau +\xi ^2}\lec N}}u_N}{X^{s,1}}}
for each dyadic $N\ge 1$, where $u_N:=P_{\shugo{\LR{\xi}\sim N}}u$.
This follows from the Cauchy-Schwarz inequality as follows:
\eqq{\norm{P_{\shugo{N\lec \LR{\tau +\xi ^2}\lec N^2}}u_N}{Y^s}&\lec \big( N^2\big) ^{1/2}\norm{P_{\shugo{N\lec \LR{\tau +\xi ^2}N^2}}u_N}{X^{s,0}}\\
&\sim \norm{P_{\shugo{N\lec \LR{\tau +\xi ^2}\lec N^2}}u_N}{X^{s+1,0}},}
\eqq{&\ttsum _{1\le M\lec N}\norm{P_{\shugo{\LR{\tau +\xi ^2}\sim M}}u_N}{Y^s}\lec \ttsum _{1\le M\lec N}M^{1/2}\norm{P_{\shugo{\LR{\tau +\xi ^2}\sim M}}u_N}{X^{s,0}}\\
&\hx \lec \ttsum _{1\le M\lec N}M^{-1/2}\norm{P_{\shugo{\LR{\tau +\xi ^2}\sim M}}u_N}{X^{s,1}}\lec \big( \ttsum _{M}M^{-1}\big) ^{1/2}\norm{P_{\shugo{\LR{\tau +\xi ^2}\lec N}}u_N}{X^{s,1}}\\
&\hx \lec \norm{P_{\shugo{\LR{\tau +\xi ^2}\lec N}}u_N}{X^{s,1}}.}

We next consider $X^{s+\th ,1-\th}\cap Y^s\hookrightarrow W^s$.
This immediately follows from the definition if $s>-\frac{1}{2}$.
For $s=-\frac{1}{2}$, it suffices to observe that the Cauchy-Schwarz inequality implies
\eqq{&\ttsum _{M\ge 1}\norm{P_{\shugo{\LR{\tau +\xi ^2}\sim M}\cap \shugo{\LR{\tau +\xi ^2}\gec \LR{\xi}^2}}u}{X^{1/2,0}}\\
&\hx \lec \ttsum _{M\ge 1}\norm{P_{\shugo{\LR{\tau +\xi ^2}\sim M}\cap \shugo{\LR{\tau +\xi ^2}\gec \LR{\xi}^2}}u}{X^{-1/2+\th ,(1-\th )/2}}\\
&\hx \lec \ttsum _{M\ge 1}M^{-(1-\th )/2}\norm{P_{\shugo{\LR{\tau +\xi ^2}\sim M}\cap \shugo{\LR{\tau +\xi ^2}\gec \LR{\xi}^2}}u}{X^{-1/2+\th ,1-\th}}\\
&\hx \lec \norm{P_{\shugo{\LR{\tau +\xi ^2}\gec \LR{\xi}^2}}u}{X^{-1/2+\th ,1-\th}}.}

The above proof also works in the case of $\th =0$.
Then, $X^{s,1}\hookrightarrow W^s$ follows from $X^{s,1}\hookrightarrow Y^s$, which is easily verified by the Cauchy-Schwarz inequality in $\tau$.
\end{proof}

The integral equation associated with the initial value problem \eqref{gB''} is
\eqq{u^\la (t)=e^{it\p _x^2}u_0^\la -i\int _0^te^{i(t-t')\p _x^2}F^\la (t')\,dt',}
where
\eqq{F^\la :=\frac{1}{2}\la ^{-2}(u^\la -\bbar{u^\la})-\frac{1}{4}\omega _\la ^2(u^\la +\bbar{u^\la})^2.}
To solve this on the interval $[-1,1]$, we take the same approach as \cite{K09-KdV} and consider the following equation:
\eqq{u^\la (t)=\psi (t)e^{it\p _x^2}u_0^\la +\Sc{I}F^\la (t),}
where $\psi :\R \to [0,1]$ is a smooth bump function satisfying $\chf{[-1,1]}\le \psi \le \chf{[-2,2]}$ and $\chf{\Omega}$ denotes the characteristic function of $\Omega$, 
\eqq{\Sc{I}F^\la &:=\frac{-i}{\sqrt{2\pi}}\psi (t)e^{it\p _x^2}\F _\xi ^{-1}\int _{\R}\ti{F}^\la (\tau ,\xi )\psi (\tau +\xi ^2)\sum _{n=1}^\I \frac{t^n}{n!}\big[ i(\tau +\xi ^2)\big] ^{n-1}\,d\tau \\
&\hx +\frac{i}{\sqrt{2\pi}}\psi (t)e^{it\p _x^2}\F _\xi ^{-1}\int _{\R}\ti{F}^\la (\tau ,\xi )\frac{1-\psi (\tau +\xi ^2)}{i(\tau +\xi ^2)}\,d\tau \\
&\hx -i\F ^{-1}_{\tau ,\xi }\Big[ \ti{F}^\la (\tau ,\xi )\frac{1-\psi (\tau +\xi ^2)}{i(\tau +\xi ^2)}\Big] .}
We observe that $\Sc{I}F^\la$ is actually an extension to $t\in \R$ of the inhomogeneous part in the integral equation on $t\in [-1,1]$.
Note that we do not put $\psi (t)$ on the third term.
This is due to the fact that it seems difficult to show the stability of our space $W^{-1/2}$ with respect to time localization, namely, $\tnorm{\psi (t)u}{W^{-1/2}}\lec \tnorm{u}{W^{-1/2}}$.
For the case $s>-\frac{1}{2}$, it turns out that our space has this property, and we may consider the usual equation
\eqq{u^\la (t)=\psi (t)e^{it\p _x^2}u_0^\la -i\psi (t)\int _0^te^{i(t-t')\p _x^2}F^\la (t')\,dt',}
similarly to the nonperiodic case \cite{KT10}.

For convenience, we define the following spacetime Fourier multipliers
\eqs{J^\sgm :=\F _\xi ^{-1}\LR{\xi}^\sgm \F _x ,\qquad \La ^\sgm :=\F _{\tau ,\xi }^{-1}\LR{\tau +\xi ^2}^\sgm \F _{t,x}}
for $\sgm \in \R$.
Homogeneous and inhomogeneous linear estimates are stated as follows.
\begin{lem}\label{lem_linear}
Let $\la \ge 1$ and $-\frac{1}{4}\ge s\ge -\frac{1}{2}$.
Then the following estimates hold with constants independent of $\la$.
\begin{enumerate}
\item $\norm{\eta (t)e^{it\p _x^2}u_0^\la}{W^s}\lec \norm{\eta}{H^1(\R )}\norm{u_0^\la}{H^s(Z_\la )}$ for any $\eta \in \Sc{S}(\R)$.
\item $\norm{\Sc{I}F^\la}{W^s}\lec \norm{\Lambda ^{-1}F^\la}{W^s}$.
\end{enumerate}
\end{lem}
\begin{proof}
(i) From Lemma~\ref{lem_embedding}, we see that
\eqq{\norm{\eta (t)e^{it\p _x^2}u_0^\la}{W^s}&\lec \norm{\eta (t)e^{it\p _x^2}u_0^\la}{X^{s,1}}=\norm{\LR{\xi}^s\LR{\tau +\xi ^2}\F \eta (\tau +\xi ^2)\F u_0^\la (\xi )}{L^2_{\tau ,\xi}}\\
&=\norm{\psi}{H^1}\norm{u_0^\la}{H^s}.}

(ii) For the first and the second terms in $\Sc{I}F^\la$ we note the support property of $\psi$ and apply the same argument for (i), obtaining the bound $\norm{\La^{-1}F^\la}{Y^s}$.
This is sufficient for the claim since $W^s\hookrightarrow Y^s$ from Lemma~\ref{lem_embedding}.
The estimate for the third term follows directly from the fact
\eqq{\Big| \ti{F}^\la (\tau ,\xi )\frac{1-\psi (\tau +\xi ^2)}{i(\tau +\xi ^2)}\Big| \lec \LR{\tau +\xi ^2}^{-1}|\ti{F}^\la (\tau ,\xi )|.\qedhere}
\end{proof}

As discussed in \cite{KT10}, linear terms in the right hand side of \eqref{gB''} is negligible when $\la$ is sufficiently large.
This can be seen from the following lemma.
\begin{lem}\label{lem_nonlinear1}
For $\la \ge 1$ and $-\frac{1}{4}\ge s\ge -\frac{1}{2}$, we have
\eqq{\norm{\La ^{-1}u^\la}{W^s}+\norm{\La ^{-1}\bbar{u^\la}}{W^s}\lec \norm{u^\la}{W^s}.}
\end{lem}
\begin{proof}
It follows from the embedding $W^s\hookrightarrow X^{s,0}~(=\La X^{s,1})~\hookrightarrow \La W^s$ given in Lemma~\ref{lem_embedding} and $\norm{\bbar{u^\la}}{X^{s,0}}=\norm{u^\la}{X^{s,0}}$.
\end{proof}

Now, we state the key bilinear estimate, which will be proved in Sections~\ref{sec_be+}--\ref{sec_be0}.
\begin{prop}\label{prop_be}
Let $\la \ge 1$ and $-\frac{1}{4}\ge s\ge -\frac{1}{2}$.
Then, we have
\eqs{\norm{\La^{-1}\omega _\la ^2(u^\la \bbar{v^\la})}{W^s}\lec C_s(\la )\norm{u^\la}{W^s}\norm{v^\la}{W^s},\\
C_s(\la )=\begin{cases}
\la ^{-2s-1/2} &\text{for}\quad -\frac{1}{4}>s\ge -\frac{1}{2},\\
\big( \log (1+\la )\big) ^{1/2} &\text{for}\quad s=-\frac{1}{4}.
\end{cases}}
If $u^\la \bbar{v^\la}$ is replaced by $u^\la v^\la$ or $\bbar{u^\la}\bbar{v^\la}$, then the above estimate holds without $C_s(\la )$.
\end{prop}

Finally, we employ the argument of Muramatu and Taoka~\cite{MT04} to prove the uniqueness of solutions in $W^s_T$.
This approach, also previously taken in \cite{K09-NLS,K09-KdV}, is effective especially if the resolution space is not a simple $X^{s,b}$ space but modified in a complicated way.
Note that the simple scaling argument used in \cite{KT10} to establish the uniqueness for the nonperiodic case with $s>-\frac{1}{2}$ cannot be applied the limiting regularity $s=-\frac{1}{2}$.

The following proposition is the key for our argument.
It was essentially proved in \cite{K09-KdV}, Lemma~4.2, employing the result of \cite{MT04}, Theorem~2.5.
However, we will give a complete proof in Appendix to keep the article self-contained.
\begin{prop}\label{prop_uniqueness}
Let $Z$ be either $\R ^d$ or $\Bo{T}^d_\la:=\R ^d/(2\pi\la \Bo{Z})^d$ for $d\ge 1$, $\la >0$.
Let $s\in \R$ and $\Sc{X}^s$ be a Banach space of functions on $\R _t\times Z_x$ with the following properties:
\begin{enumerate}
\item $\Sc{S}(\R \times Z)$ is dense in $\Sc{X}^s$,
\item $X^{s,b}(\R \times Z)\hookrightarrow \Sc{X}^s\hookrightarrow C_t(\R;H^s(Z))$ for some $b>\frac{1}{2}$,
\item $X^{s',b'}(\R \times Z)\hookrightarrow \Sc{X}^s$ for some $s'\in \R$ and $\frac{1}{2}\le b'<1$.
\end{enumerate}
Suppose that a function $u\in \Sc{X}^s$ satisfies $u(0,\cdot )=0$ in $H^s(Z)$.
Then, we have
\eq{limit_T}{\lim _{T\to +0}\norm{u}{\Sc{X}^s_T}=0.}
\end{prop}
From Lemma~\ref{lem_embedding}, we have $X^{s,1}\hookrightarrow W^s\hookrightarrow Y^s$ and $X^{s+\frac{1}{4},\frac{3}{4}}\cap Y^s\hookrightarrow W^s$.
Then, since $Y^s\hookrightarrow C_t(H^s)$ and $X^{s+\frac{1}{4},\frac{3}{4}}\hookrightarrow Y^s$, we see that our space $W^s$ ($-\frac{1}{4}\ge s\ge -\frac{1}{2}$) satisfies the above properties (i)--(iii).

We are in a position to prove the local well-posedness for \eqref{gB'}.
\begin{proof}[Proof of Theorem~\ref{thm_main} (I)]
We only consider the case $s=-\frac{1}{2}$ for simplicity.
We first show that the map
\eqq{\Phi _{\la ,u_0^{\la}}: u^\la \mapsto e^{it\p _x^2}u_0^\la -i\int _0^te^{i(t-t')\p _x^2}F^\la (t')\,dt'}
associated with the rescaled problem \eqref{gB''} is a contraction on a ball in $W^{-1/2}_1$ if $\la$ is sufficiently large and $\tnorm{u_0^\la }{H^{-1/2}}$ is sufficiently small.
By Lemmas~\ref{lem_linear}, \ref{lem_nonlinear1}, and Proposition~\ref{prop_be}, there exists $C_0>1$ independent of $\la$ such that
\eqq{\norm{\Phi _{\la ,u_0^{\la}}(u^\la )}{W^{-1/2}_1}\le C_0\Big( \norm{u_0^\la}{H^{-1/2}}+\la ^{-2}\norm{u^\la}{W_1^{-1/2}}+\la ^{1/2}\norm{u^\la}{W^{-1/2}_1}^2\Big)}
for $u^\la \in W^{-1/2}_1$ and
\eqq{&\norm{\Phi _{\la ,u_0^{\la}}(u^\la )-\Phi _{\la ,u_0^{\la}}(v^\la )}{W^{-1/2}_1}\\
&\le C_0\Big( \la ^{-2}\norm{u^\la -v^\la}{W_1^{-1/2}}+\la ^{1/2}\big( \norm{u^\la}{W^{-1/2}_1}+\norm{v^\la}{W^{-1/2}_1}\big) \norm{u^\la -v^\la}{W^{-1/2}_1}\Big)}
for $u^\la ,v^\la \in W^{-1/2}_1$.
Therefore, if $\la ^2\ge \la _0^2:=4C_0$ and $\tnorm{u_0^\la}{H^{-1/2}}\le (4C_0)^{-2}\la ^{-1/2}$, the map $\Phi _{\la ,u_0^{\la}}$ will be a contraction on $\Shugo{u^\la \in W^{-1/2}_1}{\tnorm{u^\la}{W^{-1/2}_1}\le (4C_0)^{-1}\la ^{-1/2}}$, giving a solution $u^\la$ to \eqref{gB''} for this initial datum $u_0^\la$ on the time interval $[-1,1]$.
Lipschitz continuity of the map $u_0^\la\mapsto u^\la$ is easily verified in a similar manner.

Next, consider the original problem \eqref{gB'} with initial data $u_0$ satisfying $\tnorm{u_0}{H^{-1/2}}\le r$.
If $r\le (4C_0)^{-2}\la _0^{1/2}$, then we have $\tnorm{u_0^{\la _0}}{H^{-1/2}}\le \la _0^{-1}\tnorm{u_0}{H^{-1/2}}\le (4C_0)^{-2}\la _0^{-1/2}$ and obtain a solution $u^{\la _0}$ to the $\la _0$-rescaled problem on $[-1,1]$, thus obtain a solution $u$ to \eqref{gB'} with existence time $T=\la _0^{-2}$.
If $(4C_0)^{-2}\la _0^{1/2}<r=:(4C_0)^{-2}\la (r)^{1/2}$, we solve the $\la (r)$-rescaled problem on $[-1,1]$ in the same way to obtain a solution of \eqref{gB'} with $T=\la (r)^{-2}$.

Finally, we show the uniqueness of solution.
Assume that $u$ and $v$ are solutions to \eqref{gB'} with the common data $u_0$ and the common existence time $T_0$, and that both of them belong to $W^{-1/2}_{T_0}$.
Then, $u^\la (t,x):=\la ^{-2}u(\la ^{-2}t,\la ^{-1}x)$ and $v^\la$ are solutions to the $\la $-rescaled problem \eqref{gB''} with initial data $u_0^\la (x):=\la ^{-2}u_0(\la ^{-1}x)$.
Applying Lemma~\ref{lem_nonlinear1} and Proposition~\ref{prop_be} to the integral equation, it follows for $0<T\le \min \shugo{1,\,\la ^2T_0}$ that
\eq{est_uniqueness}{\norm{u^{\la}-v^\la}{W^{-1/2}_T}\le C_0\Big( \la ^{-2}+\la ^{1/2}\norm{u^\la}{W^{-1/2}_T}+\la ^{1/2}\norm{v^\la}{W^{-1/2}_T}\Big) \norm{u^\la -v^\la}{W^{-1/2}_T}.}
From Lemma~\ref{lem_linear}, we see that
\eqq{\la ^{1/2}\norm{u^\la}{W^{-1/2}_T}&\le \la ^{1/2}\norm{u^\la -e^{it\p _x^2}u_0^\la}{W^{-1/2}_T}+\la ^{1/2}\norm{e^{it\p _x^2}u_0^\la}{W^{-1/2}_T}\\
&\le \la ^{1/2}\norm{u^\la -e^{it\p _x^2}u_0^\la}{W^{-1/2}_T}+C\la ^{-1/2}\norm{u_0}{H^{-1/2}}.}
Since $(u^\la -e^{it\p _x^2}u_0^\la )\big| _{t=0}=0$, Proposition~\ref{prop_uniqueness} implies that $\tnorm{u^\la -e^{it\p _x^2}u_0^\la}{W^{-1/2}_T}\to 0$ as $T\to 0$.
Now, for given $u_0$ we choose $\la =\la (\tnorm{u_0}{H^{-1/2}})\ge 1$ sufficiently large so that
\eqq{C_0\la ^{-2}+2C_0C\la ^{-1/2}\norm{u_0}{H^{-1/2}}\le \frac{1}{4},}
and choose $T=T(\la ,u,v)>0$ so small that
\eqq{C_0\la ^{1/2}\big( \norm{u^\la -e^{it\p _x^2}u_0^\la}{W^{-1/2}_T}+\norm{v^\la -e^{it\p _x^2}u_0^\la}{W^{-1/2}_T}\big) \le \frac{1}{4}.}
Then, \eqref{est_uniqueness} yields $\tnorm{u^{\la}-v^\la}{W^{-1/2}_T}=0$, so we conclude that $u(t)=v(t)$ for $-\la ^{-2}T\le t\le \la ^{-2}T$.
If $\la ^{-2}T=T_0$, then the claim follows.
If not, the coincidence on the whole interval $[-T_0,T_0]$ is obtained by a continuity argument.
\end{proof}


\section{Refined bilinear $L^4$ estimates for periodic functions}\label{sec_L4}

In this section, we prepare some bilinear refinement of the $L^4$ Strichartz estimate.
Let us begin with the following.
\begin{lem}[Bilinear $L^4$ estimate]\label{lem_L4}
Let $b,b'\in \R$ be such that
\eqq{b>\tfrac{1}{4},\quad b'>\tfrac{1}{4},\quad b+b'\ge \tfrac{3}{4}.}
Then, we have
\eqq{\norm{uv}{L^2_{t,x}(\R \times Z_\la )}\lec \norm{u}{X^{0,b}}\norm{v}{X^{0,b'}}.}
In LHS, $uv$ can be replaced by $\bar{u}\bar{v}$ or $u\bar{v}$.
\end{lem}

If we take $b=b'=\frac{3}{8}$, then Lemma~\ref{lem_L4} becomes equivalent to the well-known $L^4$ estimate of Bourgain~\cite{B93-1} stated as $\tnorm{u}{L^4_{t,x}}\lec \tnorm{u}{X^{0,3/8}}$.
In fact, we will always use it with $b=b'=\frac{3}{8}$ in this article.
We give a proof in the following, but a similar proof can be found in \cite{Taobook}, Proposition~2.13.
\begin{proof}
For a dyadic number $M\ge 1$, we write $u_M$ to denote the restriction of $u$ to the frequency dyadic region $\Shugo{(\tau ,\xi )\in \R \times Z^*_\la}{\LR{\tau +\xi ^2}\sim M}$.
Then, the Plancherel theorem and the triangle inequality imply
\eqq{&\norm{uv}{L^2_{t,x}}\le \ttsum _{M_1,M_2\ge 1} \norm{u_{M_1}v_{M_2}}{L^2_{t,x}}\\
&\sim \ttsum _{M_1,M_2\ge 1} \bigg( \displaystyle{\int _{\R \times Z^*_\la}}~\bigg| \displaystyle{\int _{\R \times Z^*_\la}}\ti{u}_{M_1}(\tau _1,\xi _1)\ti{v}_{M_2}(\tau -\tau _1,\xi -\xi _1)\,d\tau _1d\xi _1\bigg| ^2\,d\tau d\xi \bigg) ^{1/2}.}
(If $Z^*_\la =\Zl$, $\int _{\R \times Z^*_\la}f(\tau ,\xi)d\tau d\xi$ means $\ttsumla _{k\in \Zl}\int _\R f(\tau ,k)d\tau$.)
By Cauchy-Schwarz inequality in $(\tau _1,\xi _1)$, this is bounded by
\eqq{\ttsum _{M_1,M_2\ge 1} \bigg( \sup\limits _{(\tau ,\xi )\in \R \times Z^*_\la}\displaystyle{\int _{\big\{ \mat{\LR{\tau _1+\xi _1^2}\sim M_1\\[-2pt] \LR{\tau -\tau _1+(\xi -\xi _1)^2}\sim M_2}\big\}}}\,d\tau _1d\xi _1\bigg) ^{1/2}\norm{\ti{u}_{M_1}}{L^2_{\tau ,\xi}}\norm{\ti{v}_{M_2}}{L^2_{\tau ,\xi}}.}

Let us estimate the integral.
The quantity
\eq{id_L4_1}{(\tau _1+\xi _1^2)+(\tau -\tau _1+(\xi -\xi _1)^2)=\tau +\tfrac{\xi ^2}{2}+\tfrac{1}{2}(\xi -2\xi _1)^2}
is bounded by $\max \shugo{M_1,\,M_2}$ whenever $(\tau _1,\xi _1)$ is in the integral domain.
This implies that, for fixed $(\tau ,\xi )$, $\xi _1$ is restricted to at most two intervals of measure $O(\max \shugo{M_1,\,M_2}^{1/2})$.
On the other hand, if we also fix $\xi _1$, then $\tau _1$ is restricted to a set with its measure $O(\min \shugo{M_1,M_2})$, so we obtain
\eqq{\ttsum _{M_1,M_2\ge 1} \max \shugo{M_1,M_2}^{1/4}\min \shugo{M_1,M_2}^{1/2}\norm{\ti{u}_{M_1}}{L^2_{\tau ,\xi}}\norm{\ti{v}_{M_2}}{L^2_{\tau ,\xi}}}
as a bound of $\tnorm{uv}{L^2_{x,t}}$.

We may restrict our attention to the case $M_1\ge M_2$ by symmetry and estimate
\eqq{&\ttsum _{M,M'\ge 1} (MM')^{1/4}\cdot M^{1/2}\norm{\ti{u}_{MM'}}{L^2_{\tau ,\xi}}\norm{\ti{v}_{M}}{L^2_{\tau ,\xi}}\\
&\hx =\ttsum _{M'\ge 1} M'^{(1/4\,-b)}\sum _{M\ge 1} M^{3/4\,-b-b'}\big( (MM')^b\norm{\ti{u}_{MM'}}{L^2_{\tau ,\xi}}\big) \big( M^{b'}\norm{\ti{v}_{M}}{L^2_{\tau ,\xi}}\big) .}
Applying Cauchy-Schwarz inequality in $M$ and then summing over $M'$, we finish the proof.
\end{proof}

The following Lemmas~\ref{lem_RB1}--\ref{lem_DRB2} are modified bilinear $L^4$ estimates for periodic case which provide $\frac{1}{2}$ gain of regularity.
These estimates should be of independent interest; compare them to Lemma~\ref{lem_L4}, which has no regularity gain.

This type of smoothing effect is well-known in the nonperiodic case.
For instance, we can show that
\eqq{\norm{\iint _{\R ^2}\LR{\xi _1-(\xi -\xi _1)}^{1/2}\ti{u}(\tau _1,\xi _1)\ti{v}(\tau -\tau _1,\xi -\xi _1)\,d\tau _1d\xi _1}{L^2_{\tau ,\xi}(\R ^2)}\lec \norm{u}{X^{0,b}}\norm{v}{X^{0,b}}}
for $b>\frac{1}{2}$ (see \mbox{e.g.} Corollary~2.3 in \cite{G00p}).
In the periodic setting, such a `dispersive smoothing effect' is not available in general.
However, we can still capture the same type of smoothing effect if functions are restricted out of an `exceptional' frequency region.
On the other hand, there seems to be no way to gain regularity with respect to $x$ in this exceptional region, but such region is sufficiently small so that we can gain enough regularity with respect to $t$.
Even in the periodic case, these refined estimates enable us to make arguments close to those for the nonperiodic problem.

Estimates of similar spirit are found in the paper by Molinet (\cite{M11p}, Lemma~3.4), who treated the KdV and the modified KdV equations.
See also a result of the author (\cite{K11}, Lemma~2.5) for higher dimensional cases.
The feature of our estimates is that we specify `exceptional' frequency set where the dispersive smoothing vanishes, and separate it from unexceptional region in the estimates.

\begin{lem}\label{lem_RB1}
Let $\la \ge 1$ and
\eqq{\Ga _1:=\Shugo{(\tau ,k,\tau _1,k_1)\in (\R \times \Zl )^2}{&\Big| k_1-(k-k_1)+\sqrt{2(-\tau -\tfrac{1}{2}k^2)}\Big| \le \la ^{-1}\\
&\text{or}~\Big| k_1-(k-k_1)-\sqrt{2(-\tau -\tfrac{1}{2}k^2)}\Big| \le \la ^{-1}}.}
Then, we have
\begin{align}
&\norm{\ttfrac{1}{\la}\ttsum _{k_1\in \Zl}\displaystyle{\int _\R}\LR{k_1-(k-k_1)}^{1/2}\chf{A_1\cap \Ga _1^c}(\tau ,k,\tau _1,k_1)\ti{u}(\tau _1,k_1)\ti{v}(\tau -\tau _1,k-k_1)\,d\tau _1}{\ell ^2_kL^2_\tau}\notag \\
&\hx \hx \lec M_1^{1/2}M_2^{1/2}\norm{u}{L^2_{t,x}}\norm{v}{L^2_{t,x}},\label{est_RB1} \\
&\norm{\ttfrac{1}{\la}\ttsum _{k_1\in \Zl}\displaystyle{\int _\R}\chf{A_1\cap \Ga _1}(\tau ,k,\tau _1,k_1)\ti{u}(\tau _1,k_1)\ti{v}(\tau -\tau _1,k-k_1)\,d\tau _1}{\ell ^2_kL^2_\tau}\notag \\
&\hx \hx \lec \la ^{-1/2}\min \shugo{M_1^{1/2},\,M_2^{1/2}}\norm{u}{L^2_{t,x}}\norm{v}{L^2_{t,x}},\label{est_RB1'}
\end{align}
where $M_1,M_2\ge 1$ are dyadic numbers and
\eqq{A_1:=\Shugo{(\tau ,k,\tau _1,k_1)}{\LR{\tau _1+k_1^2}\lec M_1~\text{and}~\LR{(\tau -\tau _1)+(k-k_1)^2}\lec M_2}.}
\end{lem}

\begin{proof}
Similarly to the proof of Lemma~\ref{lem_L4}, \eqref{est_RB1} and \eqref{est_RB1'} are reduced to the estimates
\eq{est_RB1-}{\sup\limits _{(\tau ,k)\in \R \times \Zl}\displaystyle{\ttsumla _{k _1\in \Zl}}\displaystyle{\int _\R}\LR{k_1-(k-k_1)}\chf{A_1\cap \Ga _1^c}(\tau ,k,\tau _1,k_1)\,d\tau _1\lec M_1M_2}
and
\eq{est_RB1'-}{\sup\limits _{(\tau ,k)\in \R \times \Zl}\displaystyle{\ttsumla _{k _1\in \Zl}}\displaystyle{\int _\R}\chf{A_1\cap \Ga _1}(\tau ,k,\tau _1,k_1)\,d\tau _1\lec \la ^{-1}\min \shugo{M_1,\,M_2},}
respectively.

We fix $(\tau ,k)$ and exploit the identity \eqref{id_L4_1} again.
In $A_1\cap \Ga _1^c$, it follows that
\eqq{&O(\max \shugo{M_1,\,M_2})=\big| (k_1-(k-k_1))^2+2(\tau +\tfrac{1}{2}k^2)\big| \\
&\hx \ge \la ^{-1}\max \shugo{\Big| k_1-(k-k_1)+\sqrt{2(-\tau -\tfrac{1}{2}k^2)}\Big| ,\,\Big| k_1-(k-k_1)-\sqrt{2(-\tau -\tfrac{1}{2}k^2)}\Big|}\\
&\hx \ge \la ^{-1}\big| \tau +\tfrac{1}{2}k^2\big| ^{1/2}.}
Consider the following two cases.

(i) $\max \shugo{M_1,\,M_2}\gec |\tau +\tfrac{1}{2}k^2|$.
In this case, $(k_1-(k-k_1))^2$ is bounded by $O(\max \shugo{M_1,\,M_2})$, so we have
\eqq{&\ttsumla _{k _1\in \Zl}\displaystyle{\int _\R}\LR{k_1-(k-k_1)}\chf{A_1\cap \Ga _1^c}(\tau ,k,\tau _1,k_1)\,d\tau _1\\
&\hx \lec \min \shugo{M_1,\,M_2}\ttsumla _{\mat{k_1\in \Zl ;~(k_1-(k-k_1))^2\\ =O(\max \shugo{M_1,\,M_2})}}\LR{k_1-(k-k_1)}\\
&\hx \lec \min \shugo{M_1,\,M_2}\max \shugo{M_1,\,M_2}=M_1M_2.}
This proves \eqref{est_RB1-}.

(ii) $\la ^{-1}|\tau +\tfrac{1}{2}k^2|^{1/2}\lec \max \shugo{M_1,\,M_2}\ll |\tau +\tfrac{1}{2}k^2|$.
We may assume $\tau +\frac{1}{2}k^2<0$.
It follows that
\eqq{|k_1-(k-k_1)|=|2(\tau +\tfrac{1}{2}k^2)|^{1/2}+O(|\tau +\tfrac{1}{2}k^2|^{-1/2}\max \shugo{M_1,\,M_2}).}
Since $|\tau +\frac{1}{2}k^2|^{-1/2}\max \shugo{M_1,\,M_2}\gec \la ^{-1}$, the number of $k_1\in \Zl$ satisfying the above condition is comparable to $\la |\tau +\tfrac{1}{2}k^2|^{-1/2}\max \shugo{M_1,\,M_2}$, and such a $k_1$ satisfies $|k_1-(k-k_1)|\sim |\tau +\tfrac{1}{2}k^2|^{1/2}$.
Hence, we obtain the same bound, and then \eqref{est_RB1-}.

For \eqref{est_RB1'-}, it is sufficient to observe that $\Ga _1$ contains only an $O(1)$-number of $k_1$'s for each $(\tau ,k)$.
\end{proof}

\begin{lem}\label{lem_RB2}
Let $\la \ge 1$ and
\eqq{\Ga _2:=\Shugo{(\tau ,k,\tau _1,k_1)\in (\R \times \Zl )^2}{\big| \tau -k^2+2kk_1\big| \le \la ^{-1}|k|}.}
Then, we have
\begin{align}
&\norm{|k|^{1/2}\ttfrac{1}{\la}\ttsum _{k_1\in \Zl}\displaystyle{\int _\R}\chf{A_2\cap \Ga _2^c}(\tau ,k,\tau _1,k_1)\ti{u}(\tau _1,k_1)\bbar{\ti{v}(\tau _1-\tau ,k_1-k)}\,d\tau _1}{\ell ^2_kL^2_\tau}\notag \\
&\hx \hx \lec M_1^{1/2}M_2^{1/2}\norm{u}{L^2_{t,x}}\norm{v}{L^2_{t,x}},\label{est_RB2} \\
&\norm{\ttfrac{1}{\la}\ttsum _{k_1\in \Zl}\displaystyle{\int _\R}\chf{A_2\cap \Ga _2}(\tau ,k,\tau _1,k_1)\ti{u}(\tau _1,k_1)\bbar{\ti{v}(\tau _1-\tau ,k_1-k)}\,d\tau _1}{\ell ^2_kL^2_\tau}\notag \\
&\hx \hx \lec \la ^{-1/2}\min \shugo{M_1^{1/2},\,M_2^{1/2}}\norm{u}{L^2_{t,x}}\norm{v}{L^2_{t,x}},\label{est_RB2'}
\end{align}
where $M_1,M_2\ge 1$ are dyadic numbers and
\eqq{A_2:=\Shugo{(\tau ,k,\tau _1,k_1)}{\LR{\tau _1+k_1^2}\lec M_1~\text{and}~\LR{(\tau _1-\tau )+(k_1-k)^2}\lec M_2}.}
\end{lem}

\begin{proof}
We may consider only the case of $k\neq 0$ in the left-hand side of \eqref{est_RB2} and \eqref{est_RB2'}; otherwise, they are trivial.
As before, it suffices to show
\eq{est_RB2-}{\sup\limits _{(\tau ,k)\in \R \times \Zl}|k|\displaystyle{\ttsumla _{k _1\in \Zl}}\displaystyle{\int _\R}\chf{A_2\cap \Ga _2^c}(\tau ,k,\tau _1,k_1)\,d\tau _1\lec M_1M_2}
and
\eq{est_RB2'-}{\sup\limits _{(\tau ,k)\in \R \times \Zl}\displaystyle{\ttsumla _{k _1\in \Zl}}\displaystyle{\int _\R}\chf{A_2\cap \Ga _2}(\tau ,k,\tau _1,k_1)\,d\tau _1\lec \la ^{-1}\min \shugo{M_1,\,M_2}.}

Following the proof of Lemma~\ref{lem_RB1} and using the identity
\eqq{(\tau _1+k_1^2)-(\tau _1-\tau +(k_1-k)^2)=\tau -k^2+2kk_1}
instead of \eqref{id_L4_1}, we see that
\eqq{k_1=\ttfrac{\tau -k^2}{-2k}+O(\,\ttfrac{\max \shugo{M_1,M_2}}{|k|}\,),}
which yields \eqref{est_RB2-}.
\eqref{est_RB2'-} also follows similarly to \eqref{est_RB1'-} in the proof of Lemma~\ref{lem_RB1}.
\end{proof}



\begin{lem}\label{lem_DRB1}
Let $\la \ge 1$ and
\eqq{\De _1:=\Shugo{(\tau ,k,\tau _1,k_1)\in (\R \times \Zl )^2}{\big| \tau _1-k_1^2+2k_1k\big| \le \la ^{-1}|k_1|}.}
Then, we have
\begin{align}
&\norm{\ttfrac{1}{\la}\ttsum _{k_1\in \Zl}\displaystyle{\int _\R}|k_1|^{1/2}\chf{B_1\cap \De _1^c}(\tau ,k,\tau _1,k_1)\ti{u}(\tau _1,k_1)\ti{v}(\tau -\tau _1,k-k_1)\,d\tau _1}{\ell ^2_kL^2_\tau}\notag \\
&\hx \hx \lec M^{1/2}M_2^{1/2}\norm{u}{L^2_{t,x}}\norm{v}{L^2_{t,x}},\notag \\
&\norm{\ttfrac{1}{\la}\ttsum _{k_1\in \Zl}\displaystyle{\int _\R}\chf{B_1\cap \De _1}(\tau ,k,\tau _1,k_1)\ti{u}(\tau _1,k_1)\ti{v}(\tau -\tau _1,k-k_1)\,d\tau _1}{\ell ^2_kL^2_\tau}\notag \\
&\hx \hx \lec \la ^{-1/2}\min \shugo{M^{1/2},\,M_2^{1/2}}\norm{u}{L^2_{t,x}}\norm{v}{L^2_{t,x}},\notag
\end{align}
where $M,M_2\ge 1$ are dyadic numbers and
\eqq{B_1:=\Shugo{(\tau ,k,\tau _1,k_1)}{\LR{\tau +k^2}\lec M~\text{and}~\LR{(\tau -\tau _1)+(k-k_1)^2}\lec M_2}.}
\end{lem}

\begin{proof}
The claim comes down to Lemma~\ref{lem_RB2} through a duality argument.
\end{proof}

\begin{lem}\label{lem_DRB2}
Let $\la \ge 1$ and
\eqq{\De _2:=\Shugo{(\tau ,k,\tau _1,k_1)\in (\R \times \Zl )^2}{&\Big| k-(k_1-k)+\sqrt{2(-\tau _1-\tfrac{1}{2}k_1^2)}\Big| \le \la ^{-1}\\
&\text{or}~\Big| k-(k_1-k)-\sqrt{2(-\tau _1-\tfrac{1}{2}k_1^2)}\Big| \le \la ^{-1}}.}
Then, we have
\begin{align}
&\norm{\ttfrac{1}{\la}\ttsum _{k_1\in \Zl}\displaystyle{\int _\R}\LR{k-(k_1-k)}^{1/2}\chf{B_2\cap \De _2^c}(\tau ,k,\tau _1,k_1)\ti{u}(\tau _1,k_1)\bbar{\ti{v}(\tau _1-\tau ,k_1-k)}\,d\tau _1}{\ell ^2_kL^2_\tau}\notag \\
&\hx \hx \lec M^{1/2}M_2^{1/2}\norm{u}{L^2_{t,x}}\norm{v}{L^2_{t,x}},\notag \\
&\norm{\ttfrac{1}{\la}\ttsum _{k_1\in \Zl}\displaystyle{\int _\R}\chf{B_2\cap \De _2}(\tau ,k,\tau _1,k_1)\ti{u}(\tau _1,k_1)\bbar{\ti{v}(\tau _1-\tau ,k_1-k)}\,d\tau _1}{\ell ^2_kL^2_\tau}\notag \\
&\hx \hx \lec \la ^{-1/2}\min \shugo{M^{1/2},\,M_2^{1/2}}\norm{u}{L^2_{t,x}}\norm{v}{L^2_{t,x}},\notag
\end{align}
where $M,M_2\ge 1$ are dyadic numbers and
\eqq{B_2:=\Shugo{(\tau ,k,\tau _1,k_1)}{\LR{\tau +k^2}\lec M~\text{and}~\LR{(\tau _1-\tau )+(k_1-k)^2}\lec M_2}.}
\end{lem}

\begin{proof}
Again by duality, the claim is reduced to Lemma~\ref{lem_RB1}.
\end{proof}

Lemmas~\ref{lem_RB1}--\ref{lem_DRB2} can be regarded as the extension of the following nonperiodic modified bilinear $L^4$ estimates.
We note that the above estimates for $2\pi\la$ periodic functions contain $\la$ as a parameter and formally converge to the corresponding estimates stated below as $\la \to \I$.
The argument in the proof of Lemmas~\ref{lem_RB1}--\ref{lem_DRB2} can be naturally adjusted to the nonperiodic case, so we will omit the proof.
\begin{lem}\label{lem_RB_R}
The sets $A_1$, $A_2$, $B_1$, and $B_2$ are the same as in Lemmas~\ref{lem_RB1}--\ref{lem_DRB2}.
Then, we have the following estimates for spacetime functions $u$, $v$ on $\R \times \R$.
\eqq{&\norm{\int _{\R ^2}\LR{\xi _1-(\xi -\xi _1)}^{1/2}\chf{A_1}(\tau ,\xi ,\tau _1,\xi _1)\ti{u}(\tau _1,\xi _1)\ti{v}(\tau -\tau _1,\xi -\xi _1)\,d\tau _1d\xi _1}{L^2_{\tau ,\xi}}\\
+~&\norm{|\xi |^{1/2}\int _{\R ^2}\chf{A_2}(\tau ,\xi ,\tau _1,\xi _1)\ti{u}(\tau _1,\xi _1)\bbar{\ti{v}(\tau _1-\tau ,\xi _1-\xi )}\,d\tau _1d\xi _1}{L^2_{\tau ,\xi}}\\
&\hx \hx \lec M_1^{1/2}M_2^{1/2}\norm{u}{L^2_{t,x}}\norm{v}{L^2_{t,x}},}
\eqq{&\norm{\int _{\R ^2}|\xi _1|^{1/2}\chf{B_1}(\tau ,\xi ,\tau _1,\xi _1)\ti{u}(\tau _1,\xi _1)\ti{v}(\tau -\tau _1,\xi -\xi _1)\,d\tau _1d\xi _1}{L^2_{\tau ,\xi}}\\
+~&\norm{\int _{\R ^2}\LR{\xi -(\xi _1-\xi )}^{1/2}\chf{B_2}(\tau ,\xi ,\tau _1,\xi _1)\ti{u}(\tau _1,\xi _1)\bbar{\ti{v}(\tau _1-\tau ,\xi _1-\xi )}\,d\tau _1d\xi _1}{L^2_{\tau ,\xi}}\\
&\hx \hx \lec M^{1/2}M_2^{1/2}\norm{u}{L^2_{t,x}}\norm{v}{L^2_{t,x}}.}
\end{lem}


\bigskip
\section{Bilinear estimate for $s>-\frac{1}{2}$}\label{sec_be+}

In the case of $s>-\frac{1}{2}$, Proposition~\ref{prop_be} can be established by the H\"older, the Young inequalities and Bourgain's $L^4$ estimate (Lemma~\ref{lem_L4}).
Hence, in the most part of the proof there is no difference between the periodic and the nonperiodic cases.
We will concentrate on the case of $\Bo{T}$, and the same argument gives another proof of LWP on $\R$ obtained in \cite{KT10} (see also Remark~\ref{rem_diff} below).

\begin{proof}[Proof of Proposition~\ref{prop_be} for $s>-\frac{1}{2}$]
We show only the estimate for $u\bar{v}$; the other cases are treated similarly, as to be mentioned at the last part of the proof.

For a set $\Omega \subset (\R \times \Zl )^2$, define a bilinear operator $(u,v)\mapsto B_\Omega (u,v)$ by
\eqq{\ti{B_\Omega (u,v)}(\tau ,k)&:=\ttfrac{1}{2\pi}\ttfrac{1}{\la}\ttsum _{k_1\in \Zl}\displaystyle{\int _{\R}}\chf{\Omega}(\tau ,k,\tau _1,k_1)\ti{u}(\tau _1,k_1)\ti{\bar{v}}(\tau -\tau _1,k-k_1)\,d\tau _1\\
&~=\ttfrac{1}{2\pi}\ttfrac{1}{\la}\ttsum _{k_1\in \Zl}\displaystyle{\int _{\R}}\chf{\Omega}(\tau ,k,\tau _1,k_1)\ti{u}(\tau _1,k_1)\bbar{\ti{v}(\tau _1-\tau ,k_1-k)}\,d\tau _1.}
Note that $B_{(\R \times \Zl )^2}(u,v)=u\bar{v}$.

First of all, we assume $\ti{u},\ti{v}\ge 0$ without loss of generality.
Since $\omega _\la ^2$ acts as $P_{\shugo{k\neq 0}}$, we can decompose the domain of integral as $\ti{\omega _\la ^2(u\bar{v})}\le \sum\limits _{j=0}^4\ti{B_{\Omega _j}(u,v)}$,
\eqq{\Omega _0&:=\Shugo{(\tau ,k,\tau _1,k_1)\in (\R \times \Zl )^2}{|k_1|\lec 1~or~|k_1-k|\lec 1},\\
\Omega _1&:=\Shugo{(\tau ,k,\tau _1,k_1)\in (\R \times \Zl )^2}{1\lec |k_1|\lec |k_1-k|\sim |k|},\\
\Omega _2&:=\Shugo{(\tau ,k,\tau _1,k_1)\in (\R \times \Zl )^2}{1\lec |k_1-k|\lec |k_1|\sim |k|},\\
\Omega _3&:=\Shugo{(\tau ,k,\tau _1,k_1)\in (\R \times \Zl )^2}{1\lec |k|\ll |k_1|\sim |k_1-k|},\\
\Omega _4&:=\Shugo{(\tau ,k,\tau _1,k_1)\in (\R \times \Zl )^2}{\la ^{-1}\le |k|\le 1\ll |k_1|\sim |k_1-k|}.
}
In the following, we will show that
\eqq{\ttsum _{j=0}^3\norm{\La ^{-1}B_{\Omega _j}(u,v)}{W^s}&\lec \norm{u}{W^s}\norm{v}{W^s},\\
\norm{\La ^{-1}B_{\Omega _4}(u,v)}{W^s}&\lec C_s(\la )\norm{u}{W^s}\norm{v}{W^s}.}

\smallskip
\midashi{Estimate in $\Omega _0$}

In the region $|k_1|\lec 1$, for example, we note that $\LR{k}\sim \LR{k_1-k}$ and apply the Young inequality.
Also note that we may estimate the $X^{s,0}$ norm of $B_{\Omega _0}(u,v)$ with the aid of Lemma~\ref{lem_embedding}.
We have
\eqq{&\norm{B_{\Omega _0}(u,v)}{X^{s,0}}\sim \norm{\F _{t,x}B_{\Omega _0}(u,J^sv)}{\ell ^2L^2}\\
&\hx \lec \norm{\F _{t,x}P_{\shugo{|k|\lec 1}}u}{\ell ^1L^2}\norm{\ti{J^sv}}{\ell ^2L^1}\lec \norm{u}{X^{s,0}}\norm{v}{Y^s}.}
The case of $|k_1-k|\lec 1$ is treated in the same manner.

\smallskip
For the remaining cases, the algebraic relation
\eq{RE1}{L_1&:=\max \shugo{|\tau +k^2|,\,|\tau _1+k_1^2|,\,|(\tau _1-\tau )+(k_1 -k)^2|}\\
&~\ge \big| (\tau +k^2)-(\tau _1+k_1^2)+((\tau _1-\tau )+(k_1 -k)^2)\big| /3=2|k||k_1-k|/3}
will play an essential role.

\smallskip
\midashi{Estimate in $\Omega _2$}

Recall that $\LR{k}\sim \LR{k_1}$ in this region.
Consider three subregions
\eqq{\Omega _{21}&:=\Shugo{(\tau ,k,\tau _1,k_1)\in \Omega _2}{|\tau _1+k_1^2|\gec |k_1|},\\
\Omega _{22}&:=\Shugo{(\tau ,k,\tau _1,k_1)\in \Omega _2}{|(\tau _1-\tau )+(k_1 -k)^2|\gec |k_1-k|},\\
\Omega _{23}&:=\Omega _2\setminus (\Omega _{21}\cup \Omega_{22})}
separately.
In $\Omega _{21}$, we may measure $u$ in $X^{s+1,0}$.
Following the argument for $\Omega _0$ we obtain the upper bound
\eqq{\norm{B_{\Omega_{21}}(u,v)}{X^{s,0}}\lec \norm{\ti{u}}{\ell ^1L^2}\norm{v}{Y^s}.}
Since $s+1>\frac{1}{2}$, we have $\tnorm{\ti{u}}{\ell^1L^2}\lec \tnorm{u}{X^{s+1,0}}$ by the Cauchy-Schwarz inequality in $k$.
The estimate in $\Omega _{22}$ is the same.

In $\Omega _{23}$, the relation \eqref{RE1} implies that $\LR{\tau +k^2}\sim L_1\sim \LR{k}\LR{k_1-k}~(\lec \LR{k}^2)$, so we have
\eqq{\F \La ^{-1}B_{\Omega _{23}}(u,v)\lec \F B_{\Omega _{23}}(J^{-1}u,J^{-1}v).}
Also, we may estimate the $X^{s+1,0}$ norm of $\La ^{-1}B_{\Omega _{23}}(u,v)$ by the $X^{s,1}$ norm of $u$ and $v$.
We use Lemma~\ref{lem_L4} to obtain
\eqq{\norm{J^{-1}u\bbar{J^{-1}v}}{X^{s+1,0}}\sim \norm{J^su\bbar{J^{-1}v}}{L^2_{t,x}}\lec \norm{u}{X^{s,3/8}}\norm{v}{X^{-1,3/8}},}
which is an appropriate bound.


\smallskip
\midashi{Estimate in $\Omega _1$}

The argument for this case is parallel to that for $\Omega _2$.

\smallskip
\midashi{Estimate in $\Omega _3$}

Recall that $\LR{k_1}\sim \LR{k_1-k}$.
Consider three subregions
\eqq{\Omega _{31}&:=\Shugo{(\tau ,k,\tau _1,k_1)\in \Omega _3}{|\tau _1+k_1^2|\gec |k_1|},\\
\Omega _{32}&:=\Shugo{(\tau ,k,\tau _1,k_1)\in \Omega _3}{|(\tau _1-\tau )+(k_1 -k)^2|\gec |k_1-k|},\\
\Omega _{33}&:=\Omega _3\setminus (\Omega _{31}\cup \Omega_{32})}
separately.

See the estimate for $\Omega _{31}$ first.
We may measure $u$ in $X^{s+1,0}$.
Since $0>s>-\frac{1}{2}$, we can choose $1<q<2<p<\I$ such that
\eqq{-s>\ttfrac{1}{2}-\ttfrac{1}{p},\qquad 1+\ttfrac{1}{p}=\ttfrac{1}{q}+\ttfrac{1}{2},\qquad 2s+1>\ttfrac{1}{q}-\ttfrac{1}{2}.}
For such $(p,q)$, the H\"older inequality, followed by the Young and again the H\"older, implies that
\eqq{&\norm{B_{\Omega_{31}}(u,v)}{X^{s,0}}\lec \norm{\F B_{\Omega_{31}}(J^{-s}u,J^sv)}{\ell ^pL^2}\lec \norm{\F J^{-2s-1}J^{s+1}u}{\ell ^qL^2}\norm{\ti{J^sv}}{\ell ^2L^1}\\
&\hx \lec \norm{u}{X^{s+1,0}}\norm{v}{Y^s}.}
The case of $\Omega _{32}$ is almost identical.

Now, consider $\Omega _{33}$, where we again have $\LR{\tau +k^2}\sim \LR{k}\LR{k_1-k}$ from \eqref{RE1}.
We may measure $u$ and $v$ in $X^{s,1}$ or $Y^s$.
Using Lemma~\ref{lem_L4} we obtain
\eqq{\norm{B_{\Omega_{33}}(u,v)}{X^{s+1,-1}}\lec \norm{J^{-1/2}u\bbar{J^{-1/2}v}}{X^{s,0}}\lec \norm{u}{X^{-1/2,3/8}}\norm{v}{X^{-1/2,3/8}},}
while an application of the H\"older inequality implies
\eqq{\norm{\La ^{-1}B_{\Omega _{33}}(u,v)}{Y^s}\sim \norm{B_{\Omega _{33}}(J^{-1/2}u,J^{-1/2}v)}{Y^{s-1}}\lec \norm{\F (J^{-1/2}u\bbar{J^{-1/2}v})}{\ell ^\I L^1},}
which is evaluated by $\tnorm{u}{Y^{-1/2}}\tnorm{v}{Y^{-1/2}}$ from the Young inequality.

\smallskip
\midashi{Estimate in $\Omega _4$}

This is the worst case where the loss of $C_s(\la )$ occurs.
We claim that
\eq{claim_Omega4}{\norm{\La ^{-1}\proj{|k|\sim N}B_{\Omega _4}(u,v)}{W^s}\lec N^{2s+1/2}\norm{u}{W^s}\norm{v}{W^s}}
for dyadic $N\in [\la ^{-1},1]$.
The desired estimate will follow by squaring \eqref{claim_Omega4} and summing over $N$.

Fix $N$.
We imitate the argument for $\Omega _3$ and divide $\Omega _4$ into five subregions
\eqq{\Omega _{41}&:=\Shugo{(\tau ,k,\tau _1,k_1)\in \Omega _4}{|\tau _1+k_1^2|\gec |k_1|},\\
\Omega _{41}'&:=\Shugo{(\tau ,k,\tau _1,k_1)\in \Omega _4}{|k_1|\gec |\tau _1+k_1^2|\gec N|k_1|},\\
\Omega _{42}&:=\Shugo{(\tau ,k,\tau _1,k_1)\in \Omega _4}{|(\tau _1-\tau )+(k_1 -k)^2|\gec |k_1-k|},\\
\Omega _{42}'&:=\Shugo{(\tau ,k,\tau _1,k_1)\in \Omega _4}{|k_1-k|\gec |(\tau _1-\tau )+(k_1 -k)^2|\gec N|k_1-k|},\\
\Omega _{43}&:=\Omega _4\setminus (\Omega _{41}\cup \Omega_{41}'\cup \Omega _{42}\cup \Omega_{42}').}

In the region $\Omega _{41}$ or $\Omega _{41}'$, we first use the H\"older inequality in $k$ to have
\eqq{&\norm{\proj{|k|\sim N}(u\bar{v})}{X^{s,0}}\sim \norm{\proj{|k|\sim N}(J^{-s}u\bbar{J^sv})}{X^{0,0}}\lec N^{1/2}\norm{\F \proj{|k|\sim N}(J^{-s}u\bbar{J^sv})}{\ell ^\I L^2},}
and then apply the Young to obtain the bound $N^{1/2}\tnorm{u}{X^{-s,0}}\tnorm{v}{Y^s}$.
Since $N^{2s}\ge 1$ and $s+1>-s$, this is sufficient for the estimate in $\Omega _{41}$ where we may measure $u$ in $X^{s+1,0}$.
In $\Omega _{41}'$, where $u$ should be evaluated in $X^{s,1}$, we have
\eqq{\norm{P_{\shugo{\LR{\tau +k^2}\gec N\LR{k}}}u}{X^{-s,0}}\lec N^{2s}\norm{u}{X^{s,-2s}},}
and conclude \eqref{claim_Omega4}.
We employ the same argument for $\Omega _{42}$ and $\Omega_{42}'$.

In $\Omega _{43}$ we take a similar way, but now the relation \eqref{RE1} implies that $\LR{\tau +k^2}\ge |\tau +k^2|\sim N|k_1|\sim N\LR{k_1}$.
The estimate of the $X^{s+1,0}$ norm is
\eqq{&\norm{\proj{|k|\sim N}B_{\Omega _{43}}(u,v)}{X^{s+1,-1}}\lec N^{1/2}\norm{\F \La ^{-1}\proj{|k|\sim N}B_{\Omega _{43}}(u,v)}{\ell ^\I L^2}\\
&\hx \lec N^{1/2}N^{2s}\norm{\F \La ^{-1-2s}(J^su\bbar{J^sv})}{\ell ^\I L^2}\lec N^{1/2}N^{2s}\norm{u}{X^{s,0}}\norm{v}{Y^s}.}
The $Y^s$ norm can be treated similarly and estimated by $N^{1/2}N^{2s}\tnorm{u}{Y^s}\tnorm{v}{Y^s}$, therefore \eqref{claim_Omega4} also follows in this case.

\smallskip
All the above argument works in the case of $uv$ and $\bar{u}\bar{v}$ with some trivial modification.
We use the algebraic relation
\eq{RE2}{L_2&:=\max \shugo{|\tau +k^2|,\,|\tau _1+k_1^2|,\,|(\tau -\tau _1)+(k-k_1)^2|}\\
&~\ge \big| (\tau +k^2)-(\tau _1+k_1^2)-((\tau -\tau _1)+(k-k_1)^2)\big| /3=2|k_1||k-k_1|/3}
for the $uv$ case and
\eq{RE3}{L_3&:=\max \shugo{|\tau +k^2|,\,|-\tau _1+k_1^2|,\,|(\tau _1-\tau )+(k_1 -k)^2|}\\
&~\ge \big| (\tau +k^2)+(-\tau _1+k_1^2)+((\tau _1-\tau )+(k_1 -k)^2)\big| /3=(k^2+k_1^2+(k_1-k)^2)/3}
for the $\bar{u}\bar{v}$ case instead of \eqref{RE1}.
The bilinear operator $B_\Omega (u,v)$ is also replaced by
\eqq{\ti{B'_\Omega (u,v)}(\tau ,k)&:=\ttfrac{1}{2\pi}\ttfrac{1}{\la}\ttsum _{k_1\in \Zl}\displaystyle{\int _{\R}}\chf{\Omega}(\tau ,k,\tau _1,k_1)\ti{u}(\tau _1,k_1)\ti{v}(\tau -\tau _1,k-k_1)\,d\tau _1}
for $uv$ and
\eqq{\ti{B''_\Omega (u,v)}(\tau ,k)&:=\ttfrac{1}{2\pi}\ttfrac{1}{\la}\ttsum _{k_1\in \Zl}\displaystyle{\int _{\R}}\chf{\Omega}(\tau ,k,\tau _1,k_1)\ti{\bar{u}}(\tau _1,k_1)\ti{\bar{v}}(\tau -\tau _1,k-k_1)\,d\tau _1\\
&~=\ttfrac{1}{2\pi}\ttfrac{1}{\la}\ttsum _{k_1\in \Zl}\displaystyle{\int _{\R}}\chf{\Omega}(\tau ,k,\tau _1,k_1)\bbar{\ti{u}(-\tau _1,-k_1)}\bbar{\ti{v}(\tau _1-\tau ,k_1-k)}\,d\tau _1}
for $\bar{u}\bar{v}$.
In fact, situation is much better than the case of $u\bar{v}$ and there is no loss in $\la$ from the region $|k|\le 1$ (there is no need for separating $\Omega _3$ and $\Omega _4$).
\end{proof}

\begin{rem}\label{rem_diff}
Concerning the bilinear estimate for the nonperiodic case, the only difference from the above proof appears in the estimate inside $\Omega _4$; we also have to consider the case $|\xi |< \la ^{-1}$.
We still have \eqref{claim_Omega4} for dyadic numbers $N<\la ^{-1}$.
Then, noting that $\omega _\la ^2\sim \la ^2N^2$ for frequencies $|\xi |\sim N<\la ^{-1}$, we have
\eqq{\norm{\La ^{-1}\proj{|k|\sim N}\omega _\la ^2B_{\Omega _4}(u,v)}{W^s}\lec \la ^2N^{2s+5/2}\norm{u}{W^s}\norm{v}{W^s}.}
Since $2s+\frac{5}{2}>0$, we can sum up over $0<N<\la ^{-1}$ and reach the conclusion.
\end{rem}


\bigskip
\section{Bilinear estimate for $s=-\frac{1}{2}$}\label{sec_be0}

If we try to apply the above proof of Proposition~\ref{prop_be} to the case of $s=-\frac{1}{2}$, the logarithmic divergences will occur in several parts of the proof.
To overcome these divergences, we shall exploit modified bilinear $L^4$ estimates (Lemmas~\ref{lem_RB1}--\ref{lem_DRB2}) which provide $\frac{1}{2}$ gain of regularity.

Again, we will focus on the case of $\Bo{T}_\la$; the nonperiodic case is easier to treat, and it suffices to use Lemma~\ref{lem_RB_R} instead of Lemmas~\ref{lem_RB1}--\ref{lem_DRB2} and then modify the estimate in low frequency (see Remark~\ref{rem_diff}).

\begin{proof}[Proof of Proposition~\ref{prop_be} for $s=-\frac{1}{2}$]
We first establish the bilinear estimate for $u\bar{v}$ by modifying the proof for the case of $s>-\frac{1}{2}$.
Notations are the same as before.

\smallskip
\midashi{Estimate in $\Omega _0$}

The previous proof works also in the present case.

\smallskip
\midashi{Estimate in $\Omega _2$}

Recall the previous division
\eqq{\Omega _{21}&:=\Shugo{(\tau ,k,\tau _1,k_1)\in \Omega _2}{|\tau _1+k_1^2|\gec |k_1|},\\
\Omega _{22}&:=\Shugo{(\tau ,k,\tau _1,k_1)\in \Omega _2}{|(\tau _1-\tau )+(k_1 -k)^2|\gec |k_1-k|},\\
\Omega _{23}&:=\Omega _2\setminus (\Omega _{21}\cup \Omega_{22}).}

We first see that the estimate in $\Omega _{23}$ is completely the same as before.
Observe that $\LR{k}\lec \LR{\tau +k^2}\sim \LR{k}\LR{k_1-k}\lec \LR{k}^2$ in this region and $u\bar{v}$ is estimated in $X^{1/2,-1}$.

In $\Omega _{21}$, the same proof is not applicable because of the criticality.
However, since $\LR{k}\sim \LR{k_1}$ in this case, it suffices to show
\eqq{\norm{B_{\Omega _{21}}(u_N,v)}{X^{-1/2,0}}\lec \norm{u_N}{X^{1/2,0}}\norm{v}{Y^{-1/2}}}
for each dyadic $N\ge 1$, where $u_N:=P_{\shugo{\LR{k}\sim N}}u$.
In fact, if we show this, then it follows that
\eqq{&\norm{B_{\Omega _{21}}(u,v)}{X^{-1/2,0}}^2\sim \ttsum _{N\ge 1}\norm{P_{\shugo{\LR{k}\sim N}}B_{\Omega _{21}}(u,v)}{X^{-1/2,0}}^2\\
&\hx \lec \ttsum _{N\ge 1}\norm{B_{\Omega _{21}}(u_N,v)}{X^{-1/2,0}}^2\lec \ttsum _{N\ge 1}\norm{u_N}{X^{1/2,0}}^2\cdot \norm{v}{Y^{-1/2}}^2\sim \norm{u}{X^{1/2,0}}^2\norm{v}{Y^{-1/2}}^2,}
which is the desired estimate.
Since the Cauchy-Schwarz inequality implies
\eqq{\tnorm{\ti{u_N}}{\ell ^1L^2}\lec N^{1/2}\tnorm{\ti{u_N}}{\ell ^2L^2}\sim \tnorm{u_N}{X^{1/2,0}},}
the argument for $s>-\frac{1}{2}$ is now applicable.

The case of $\Omega _{22}$ needs a little more attention, since we cannot decompose $v$ into dyadic pieces as above.
We now assume by the previous case that $\LR{\tau _1+k_1^2}\ll \LR{k_1}$, thus $u$ should be measured in $X^{-1/2,1}$.
Consider the following three subsets of $\Omega _{22}$:
\eqq{\Omega _{22a}&:=\Omega _{22}\cap \shugo{\LR{\tau _1+k_1^2}\ll \LR{k_1},\, \LR{k_1-k}\lec \LR{\tau _1-\tau +(k_1-k)^2}\ll \LR{k}\LR{k_1-k}},\\
\Omega _{22b}&:=\Omega _{22}\cap \shugo{\LR{\tau _1+k_1^2}\ll \LR{k_1},\, \LR{\tau _1-\tau +(k_1-k)^2}\sim \LR{k}\LR{k_1-k}},\\
\Omega _{22c}&:=\Omega _{22}\cap \shugo{\LR{\tau _1+k_1^2}\ll \LR{k_1},\, \LR{\tau _1-\tau +(k_1-k)^2}\gg \LR{k}\LR{k_1-k}}.
}
Then, the $W^{-1/2}$ norm of $v$ is bounded from below by $\tnorm{v}{X^{1/2,0}}$ in $\Omega _{22a}$ and comparable to
\eqq{\ttsum _{M_2\ge 1}\norm{v_{M_2}}{X^{1/2,0}}+\norm{v}{Y^{-1/2}}}
in $\Omega _{22b}\cup \Omega _{22c}$, where $M_2$ is dyadic and $v_{M_2}:=P_{\shugo{\LR{\tau +k^2}\sim M_2}}v$.

The estimate in $\Omega _{22a}$ is similar to that for $\Omega _{23}$.
In fact, it also holds that $\LR{\tau +k^2}\sim \LR{k}\LR{k_1-k}$.
Then, the $X^{1/2,-1}$ norm of $B_{\Omega _{22a}}(u,v)$ is bounded by $\tnorm{J^{-1/2}u\bbar{J^{-1}v}}{L^2_{t,x}}$ similarly, which is in term estimated with the Young inequality as
\eqq{\norm{\ti{J^{-1/2}u}}{\ell ^2L^1}\norm{\ti{J^{-1}v}}{\ell ^1L^2}\lec \norm{u}{Y^{-1/2}}\norm{v}{X^{-1/2+\e ,0}}}
for any $\e >0$.

For $\Omega _{22b}$, we will use Lemma~\ref{lem_RB2}, one of modified versions of Lemma~\ref{lem_L4} stated in the beginning of this section.
It suffices to evaluate the $X^{-1/2,0}$ norm of $B_{\Omega _{22b}}(u,v)$ in the following way:
\eqq{\norm{B_{\Omega _{22b}}(u_N,v)}{X^{-1/2,0}}\lec \norm{u_N}{X^{-1/2,1}}\ttsum _{M_2\ge 1}\norm{v_{M_2}}{X^{1/2,0}}}
for any dyadic $N\ge 1$.

In the region $\Omega _{22b}\cap \Ga _2^c$, we decompose $u$ and $v$ into dyadic frequency pieces in $\tau +k^2$ and apply Lemma~\ref{lem_RB2} to each one, obtaining
\eqq{&\norm{B_{\Omega _{22b}\cap \Ga _2^c}(u_N,v)}{X^{-1/2,0}}\lec N^{-1/2}\ttsum _{M_1\ge 1}\ttsum _{M_2\ge 1}\norm{B_{\Omega _{22b}\cap \Ga _2^c}(P_{\shugo{\LR{\tau +k^2}\sim M_1}}u_N,v_{M_2})}{L^2_{t,x}}\\
&\hx \lec N^{-1}\ttsum _{M_1\ge 1}\ttsum _{M_2\ge 1} M_1^{1/2}M_2^{1/2}\norm{P_{\shugo{\LR{\tau +k^2}\sim M_1}}u_N}{L^2_{t,x}}\norm{P_{\shugo{N\LR{k}\sim M_2}}v_{M_2}}{L^2_{t,x}}
.}
From the Cauchy-Schwarz inequality, we see that
\eqq{&N^{-1/2}\ttsum _{M_1\ge 1}M_1^{1/2}\norm{P_{\shugo{\LR{\tau +k^2}\sim M_1}}u_N}{L^2_{t,x}}\\
&\hx \lec \Big( \ttsum _{M_1\ge 1}M_1^{-1}\Big) ^{1/2}\Big( \ttsum _{M_1\ge 1}N^{-1}M_1^2\norm{P_{\shugo{\LR{\tau +k^2}\sim M_1}}u_N}{L^2_{t,x}}^2\Big) ^{1/2}\sim \norm{u_N}{X^{-1/2,1}},}
while we have
\eqq{N^{-1/2}\ttsum _{M_2\ge 1}M_2^{1/2}\norm{P_{\shugo{N\LR{k}\sim M_2}}v_{M_2}}{L^2_{t,x}}\lec \ttsum _{M_2\ge 1}\norm{v_{M_2}}{X^{1/2,0}},}
which concludes the estimate.

In $\Omega _{22b}\cap \Ga _2$, we decompose only $u$ and apply Lemma~\ref{lem_RB2}, then
\eqq{&\norm{B_{\Omega _{22b}\cap \Ga _2}(u_N,v)}{X^{-1/2,0}}\lec N^{-1/2}\ttsum _{M_1\ge 1}\norm{B_{\Omega _{22b}\cap \Ga _2}(P_{\shugo{\LR{\tau +k^2}\sim M_1}}u_N,v)}{L^2_{t,x}}\\
&\hx \lec \la ^{-1/2}N^{-1/2}\ttsum _{M_1\ge 1}M_1^{1/2}\norm{P_{\shugo{\LR{\tau +k^2}\sim M_1}}u_N}{L^2_{t,x}}\cdot \norm{v}{L^2_{t,x}}\lec \norm{u_N}{X^{-1/2,1}}\norm{v}{X^{1/2,0}},}
as desired.
(Since $\ti{v}$ is restricted to the region $\LR{\tau +k^2}\sim N\LR{k}$, we can apply the second estimate of Lemma~\ref{lem_RB2} with $M_2=N^2$.
However, we see from the proof of Lemma~\ref{lem_RB2} that such restriction is actually not needed.)

\begin{rem}
Since $\LR{\tau _1-\tau +(k_1-k)^2}\sim L_1$ is the biggest in $\Omega _{22b}$, it seems natural to apply Lemma~\ref{lem_DRB1} (with $u$ and $v$ replaced by $\bar{v}$ and $u$, respectively) rather than Lemma~\ref{lem_RB2}.
However, Lemma~\ref{lem_DRB1} will provide only $|k_1-k|^{1/2}$ gain of regularity, which is not enough in this case.
We have used Lemma~\ref{lem_RB2} to obtain $N^{1/2}$ gain of regularity, but then we need a little stronger structure than the simple $X^{1/2,0}$ in order to sum up the dyadic frequency pieces in the estimate without any loss of regularity.
It is only here that such `$\ell ^1$-Besov' structure in the $W^{-1/2}$ norm is essentially needed.
\end{rem}

Finally, we treat $\Omega _{22c}$.
It holds from \eqref{RE1} that $\LR{\tau _1-\tau +(k_1-k)^2}\sim \LR{\tau +k^2} \gg \LR{k}\LR{k_1-k}$, which implies
\eqq{\F \La ^{-1}B_{\Omega _{22c}}(u,v)\ll \F B_{\Omega _{22c}}(J^{-1}u,J^{-1}v),}
similarly to the case of $\Omega _{23}$.
It is then enough to evaluate
\eqq{&\ttsum _{M\ge 1}\norm{\proj{\LR{\tau +k^2}\sim M}B_{\Omega _{22c}}(u,v)}{X^{1/2,-1}}+\norm{\La ^{-1}B_{\Omega _{22c}}(u,v)}{Y^{-1/2}}\\
&\hx \lec \ttsum _{M\ge 1}\norm{B_{\Omega _{22c}}(J^{-1/2}u,J^{-1}\proj{\LR{\tau +k^2}\sim M}v)}{X^{0,0}}+\norm{\F B_{\Omega _{22c}}(J^{-3/2}u,J^{-1}v)}{\ell ^2L^1}.}
Applying the Young and the H\"older inequalities to each term, we have the bound
\eqq{&\norm{\ti{J^{-1/2}u}}{\ell ^2L^1}\ttsum _{M\ge 1}\norm{\F J^{-1}\proj{\LR{\tau +k^2}\sim M}v}{\ell ^1L^2}+\norm{\ti{J^{-3/2}u}}{\ell ^1L^1}\norm{\ti{J^{-1}v}}{\ell ^2L^1}\\
&\hx \lec \norm{u}{Y^{-1/2}}\ttsum _{M\ge 1}\norm{\proj{\LR{\tau +k^2}\sim M}v}{X^{-1/2+\e ,0}}+\norm{u}{Y^{-1+\e}}\norm{v}{Y^{-1}}}
for any $\e >0$, which easily implies the claim.

\smallskip
\midashi{Estimate in $\Omega _1$}

Now $|k_1-k|$ is comparable to $|k|$, and $|k_1|$ can be very small.
If we consider a similar decomposition
\eqq{\Omega _{11}&:=\Shugo{(\tau ,k,\tau _1,k_1)\in \Omega _1}{|\tau _1-\tau +(k_1 -k)^2|\gec |k_1-k|},\\
\Omega _{12}&:=\Shugo{(\tau ,k,\tau _1,k_1)\in \Omega _1}{|\tau _1+k_1^2|\gec |k_1|},\\
\Omega _{13}&:=\Omega _1\setminus (\Omega _{11}\cup \Omega_{12}),}
then $\Omega _{13}$ is treated in the same manner as $\Omega _{23}$.
In $\Omega _{11}$, we imitate the estimate for $\Omega _{21}$, in turn decomposing $v$ into dyadic pieces in $k$.

For $\Omega _{12}$, we again perform a similar decomposition
\eqq{\Omega _{12a}&:=\Omega _{12}\cap \shugo{\LR{\tau _1-\tau +(k_1-k)^2}\ll \LR{k_1-k},\, \LR{k_1}\lec \LR{\tau _1+k_1^2}\ll \LR{k}\LR{k_1-k}},\\
\Omega _{12b}&:=\Omega _{12}\cap \shugo{\LR{\tau _1-\tau +(k_1-k)^2}\ll \LR{k_1-k},\, \LR{\tau _1+k_1^2}\sim \LR{k}\LR{k_1-k}},\\
\Omega _{12c}&:=\Omega _{12}\cap \shugo{\LR{\tau _1-\tau +(k_1-k)^2}\ll \LR{k_1-k},\, \LR{\tau _1+k_1^2}\gg \LR{k}\LR{k_1-k}}.
}

The argument for $\Omega _{12a}$ or $\Omega _{12c}$ is the same, so we focus on the case of $\Omega _{12b}$.
If we localize $v$ (and thus $u\bar{v}$) to the frequency $\LR{k}\sim N$, we see that $u$ has very high modulation $\LR{\tau _1+k_1^2}\sim N^2$, and that it is risky to employ Lemma~\ref{lem_RB2} as for $\Omega _{22b}$.
Rather, it is natural here to use Lemma~\ref{lem_DRB2}, and fortunately it will provide enough gain of regularity.
It will turn out that the stronger (Besov) structure of $W^{-1/2}$ is not necessary here.

Lemma~\ref{lem_DRB2} yields $\LR{k-(k_1-k)}^{1/2}$ gain of regularity, so we further divide $\Omega _{12b}$ as
\eqq{\Omega _{12b-0}&:=\Shugo{(\tau ,k,\tau _1,k_1)\in \Omega _{12b}}{\LR{k-(k_1-k)}\ll \LR{k}},\\
\Omega _{12b-1}&:=\Shugo{(\tau ,k,\tau _1,k_1)\in \Omega _{12b}}{\LR{k-(k_1-k)}\sim \LR{k}}.}

In $\Omega _{12b-0}$, we can exploit the property $\LR{k}\sim \LR{k_1}\sim \LR{k_1-k}$.
Then, the estimate is much easier; for instance, the argument for $\Omega _{11}$ or $\Omega _{21}$ is also sufficient here.

We next see the estimate in $\Omega _{12b-1}\cap \De _2^c$, where $\De _2$ is as in Lemma~\ref{lem_DRB2}.
Since $\LR{\tau +k^2}\lec \LR{k}\LR{k_1-k}\sim \LR{k}^2$, we do not need to control the $Y^{-1/2}$ norm of $\La ^{-1}(u\bar{v})$.
By Lemma~\ref{lem_embedding} ($\th =\frac{3}{4}$), it suffices to estimate
\eqq{&\norm{\proj{\LR{\tau +k^2}\lec \LR{k}}B_{\Omega _{12b-1}\cap \De _2^c}(u,v)}{X^{-1/2,0}}\\
&+\norm{\proj{\LR{k}\lec \LR{\tau +k^2}\lec \LR{k}^2}B_{\Omega _{12b-1}\cap \De _2^c}(u,v)}{X^{1/4,-3/4}},}
thus we will show
\eqq{&N^{-1/2}\ttsum _{M\lec N}\norm{\proj{\LR{\tau +k^2}\sim M}B_{\Omega _{12b-1}\cap \De _2^c}(u,v_N)}{L^2_{t,x}}\\
&+N^{1/4}\ttsum _{N\lec M\lec N^2}M^{-3/4}\norm{\proj{\LR{\tau +k^2}\sim M}B_{\Omega _{12b-1}\cap \De _2^c}(u,v_N)}{L^2_{t,x}}\lec \norm{u}{X^{1/2,0}}\norm{v_N}{X^{-1/2,1}}}
for each dyadic $N\ge 1$, where $v_N:=\proj{\LR{k}\sim N}v$.
Note that Lemma~\ref{lem_DRB2} now produces the $N^{1/2}$ gain of regularity.
Decomposing $v$ with respect to $\tau +k^2$ and applying Lemma~\ref{lem_DRB2}, we bound the left-hand side by
\eqq{&\Big( N^{-1}\ttsum _{M\lec N}M^{1/2}+N^{-1/4}\ttsum _{N\lec M\lec N^2}M^{-1/4}\Big) \norm{u}{L^2_{t,x}}\ttsum _{M_2\ge 1}M_2^{1/2}\norm{\proj{\LR{\tau +k^2}\sim M_2}v_N}{L^2_{t,x}}\\
&\hx \lec \norm{u}{L^2_{t,x}}\cdot N^{-1/2}\ttsum _{M_2\ge 1}M_2^{1/2}\norm{\proj{\LR{\tau +k^2}\sim M_2}v_N}{L^2_{t,x}}\lec \norm{u}{X^{1/2,0}}\norm{v_N}{X^{-1/2,1}},}
as desired.

Finally, in $\Omega _{12b-1}\cap \De _2$ we decompose $v$ again and use Lemma~\ref{lem_DRB2} to obtain
\eqq{\norm{B_{\Omega _{12b-1}\cap \De _2}(u,v_N)}{X^{-1/2,0}}\lec \la ^{-1/2}\norm{u}{L^2_{t,x}}N^{-1/2}\ttsum _{M_2\ge 1}M_2^{1/2}\norm{\proj{\LR{\tau +k^2}\sim M_2}v_N}{L^2_{t,x}},}
which is sufficient.


\smallskip
\midashi{Estimate in $\Omega _3$}

We begin with the same division of domain as before:
\eqq{\Omega _{31}&:=\Shugo{(\tau ,k,\tau _1,k_1)\in \Omega _3}{|\tau _1+k_1^2|\gec |k_1|},\\
\Omega _{32}&:=\Shugo{(\tau ,k,\tau _1,k_1)\in \Omega _3}{|\tau _1-\tau +(k_1 -k)^2|\gec |k_1-k|},\\
\Omega _{33}&:=\Omega _3\setminus (\Omega _{31}\cup \Omega_{32}).}

$\Omega _{31}$ and $\Omega _{32}$ are almost symmetric.
Consider $\Omega _{31}$.
We first deal with the $Y^{-1/2}$ norm, which is easily handled with the H\"older inequality followed by the Young:
\eqq{&\norm{\La ^{-1}\proj{\LR{\tau +k^2}\gg \LR{k}^2}B_{\Omega _{31}}(u,v)}{Y^{-1/2}}\lec \norm{\La ^{-3/4}(J^{1/2}u\bbar{J^{-1/2}v})}{Y^{-1}}\\
&\hx \lec \norm{\F (J^{1/2}u\bbar{J^{-1/2}v})}{\ell ^\I L^2}\lec \norm{u}{X^{1/2,0}}\norm{v}{Y^{-1/2}}.}

To estimate the remainder of the $W^{-1/2}$ norm, a further decomposition is required:
\eqq{\Omega _{31a}&:=\Shugo{(\tau ,k,\tau _1,k_1)\in \Omega _{31}}{|\tau _1-\tau +(k_1 -k)^2|\lec |k_1-k|},\\
\Omega _{31b}&:=\Shugo{(\tau ,k,\tau _1,k_1)\in \Omega _{31}}{|\tau _1-\tau +(k_1 -k)^2|\gec |k_1-k|}.}

In $\Omega _{31a}$ we have to measure $v$ in $X^{-1/2,1}$, so the $\LR{k_1-k}^{1/2}$ gain of regularity is essential, which we can generate from Lemma~\ref{lem_DRB2}.
Note that $\LR{k-(k_1-k)}\sim \LR{k_1-k}$ in $\Omega _3$.
For the estimate in $\Omega _{31a}\cap \De _2^c$, we use Lemma~\ref{lem_embedding} with $\th >\frac{1}{2}$ and Lemma~\ref{lem_DRB2} as follows:
\eqq{&\norm{B_{\Omega _{31a}\cap \De _2^c}(u,v)}{X^{\th -1/2,-\th}}\lec \ttsum _{M\ge 1}\norm{\proj{\LR{\tau +k^2}\sim M}B_{\Omega _{31a}\cap \De _2^c}(J^{\th -1/2}u,v)}{X^{0,-\th}}\\
&\hx \lec \ttsum _{M\ge 1}\ttsum _{M_2\ge 1}M^{1/2-\th}M_2^{1/2}\norm{u}{X^{\th -1/2,0}}\norm{\proj{\LR{\tau +k^2}\sim M_2}v}{X^{-1/2,0}}\lec \norm{u}{X^{1/2,0}}\norm{v}{X^{-1/2,1}}.}
At the last inequality we have used the Cauchy-Schwarz inequality in $M_2$.
In $\Omega _{31a}\cap \De _2$, we have
\eqq{&\norm{B_{\Omega _{31a}\cap \De _2}(u,v)}{X^{-1/2,0}}\lec \norm{B_{\Omega _{31a}\cap \De _2}(J^{1/2}u,J^{-1/2}v)}{X^{0,0}}\\
&\hx \lec \la ^{-1/2}\norm{u}{X^{1/2,0}}\ttsum _{M_2\ge 1}M_2^{1/2}\norm{\proj{\LR{\tau +k^2}\sim M_2}v}{X^{-1/2,0}}\lec \norm{u}{X^{1/2,0}}\norm{v}{X^{-1/2,1}}.}

It remains to treat $\Omega _{31b}$.
Here, we may measure both $u$ and $v$ in $X^{1/2,0}$.
From Lemma~\ref{lem_embedding} with $\th =\frac{1}{2}+\e$ ($0<\e \ll 1$), it suffices to bound
\eqq{\norm{B_{\Omega _{31b}}(u,v)}{X^{\e ,-1/2-\e}}\lec \norm{B_{\Omega _{31b}}(J^{\e}u,v)}{X^{0,-1/2-\e}}.}
Using the H\"older and the Young inequalities, we evaluate the above by
\eqq{\norm{\F B_{\Omega _{31b}}(J^{\e}u,v)}{\ell ^2L^\I}\lec \norm{\ti{J^\e u}}{\ell ^{4/3}L^2}\norm{\ti{v}}{\ell ^{4/3}L^2}\lec \norm{u}{X^{1/2,0}}\norm{v}{X^{1/2,0}},}
as required.

For $\Omega _{32}$ we just have to use Lemma~\ref{lem_DRB1} (with $u$ and $v$ replaced with $\bar{v}$ and $u$, respectively) instead of Lemma~\ref{lem_DRB2} and follow the above argument for $\Omega _{31}$.


Now, we treat the remaining region $\Omega _{33}$, where we have to estimate
\eqq{\ttsum _{M\ge 1}\norm{\proj{\LR{\tau +k^2}\sim M}B_{\Omega _{33}}(u,v)}{X^{1/2,-1}}+\norm{\La ^{-1}B_{\Omega _{33}}(u,v)}{Y^{-1/2}}.}
The $Y^{-1/2}$ norm is treated in the same way as for the case of $s>-\frac{1}{2}$.
Since $\LR{k_1}\sim \LR{k_1-k}$ and $\LR{\tau +k^2}\sim \LR{k}\LR{k_1-k}\gec \LR{k_1-k}$, after decomposing $u$ and $v$ in $k$ we only have to show that
\eqq{\ttsum _{M\gec N}\Big( \displaystyle{\ttfrac{M}{N}}\Big)^{1/2} M^{-1}\norm{\proj{\LR{\tau +k^2}\sim M}(u_N\bbar{v_N})}{X^{0,0}}\lec N^{-1}\norm{u_N}{X^{0,1}}\norm{v_N}{X^{0,1}}}
for each dyadic $N\ge 1$, where $u_N:=\proj{\LR{k}\sim N}u$ and similarly for $v_N$.
This estimate follows easily from Lemma~\ref{lem_L4}.


\smallskip
\midashi{Estimate in $\Omega _4$}

We shall prove that
\eq{claim_Omega4'}{\norm{\La ^{-1}\proj{|k|\sim N}B_{\Omega _4}(u,v)}{W^{-1/2}}\lec N^{-1/2}\norm{u}{W^{-1/2}}\norm{v}{W^{-1/2}}}
for dyadic $N\in [\la ^{-1},1]$.
The desired estimate will follow by summing \eqref{claim_Omega4} over $N$.

We fix $N$ and again divide $\Omega _4$ as follows:
\eqq{\Omega _{41}&:=\Shugo{(\tau ,k,\tau _1,k_1)\in \Omega _4}{|\tau _1+k_1^2|\gec |k_1|},\\
\Omega _{41}'&:=\Shugo{(\tau ,k,\tau _1,k_1)\in \Omega _4}{|k_1|\gec |\tau _1+k_1^2|\gec N|k_1|},\\
\Omega _{42}&:=\Shugo{(\tau ,k,\tau _1,k_1)\in \Omega _4}{|(\tau _1-\tau )+(k_1 -k)^2|\gec |k_1-k|},\\
\Omega _{42}'&:=\Shugo{(\tau ,k,\tau _1,k_1)\in \Omega _4}{|k_1-k|\gec |(\tau _1-\tau )+(k_1 -k)^2|\gec N|k_1-k|},\\
\Omega _{43}&:=\Omega _4\setminus (\Omega _{41}\cup \Omega_{41}'\cup \Omega _{42}\cup \Omega_{42}'),}
then the previous argument for $s>-\frac{1}{2}$ also works in each region, except for the estimate of the $X^{1/2,0}$ norm in $\Omega _{43}$.

To conclude the proof, it is sufficient to decompose $u$ and $v$ into dyadic pieces in $k$ and show that
\eqq{\ttsum _{M\ge 1}\norm{\proj{\LR{\tau +k^2}\sim M}\proj{|k|\sim N}B_{\Omega _{43}}(u_{N_1},v_{N_1})}{X^{1/2,-1}}\lec N^{-1/2}\norm{u_{N_1}}{X^{-1/2,1}}\norm{v_{N_1}}{X^{-1/2,1}}}
for dyadic $N_1\ge 1$.
Observe again that $\LR{\tau +k^2}\ge |\tau +k^2|\sim N|k_1|\sim NN_1$, which implies
\eqq{&\ttsum _{M\ge 1}\norm{\proj{\LR{\tau +k^2}\sim M}\proj{|k|\sim N}B_{\Omega _{43}}(u_{N_1},v_{N_1})}{X^{1/2,-1}}\\
&\hx \sim \norm{\proj{\LR{\tau +k^2}\sim \max \shugo{1,NN_1}}\proj{|k|\sim N}B_{\Omega _{43}}(u_{N_1},v_{N_1})}{X^{1/2,-1}}.}
Therefore, the summation over $M$ essentially consists of just one $M$, and we may estimate the usual $X^{1/2,0}$ norm instead of the stronger Besov-type $X^{1/2,0}$ norm.
Now, the previous argument for $s>-\frac{1}{2}$ works.

\smallskip
When the nonlinearity $u\bar{v}$ is replaced by $uv$ or $\bar{u}\bar{v}$, we use the algebraic relation \eqref{RE2} for $uv$ or \eqref{RE3} for $\bar{u}\bar{v}$, and also replace $B_\Omega (u,v)$ with $B'_\Omega (u,v)$ or $B''_\Omega (u,v)$, respectively.
Note that for these nonlinearities the two cases of $\Omega _1$ and $\Omega _2$ are completely symmetric and we just have to consider any one of them, and also that similarly to the case of $s>-\frac{1}{2}$, we do not need to separate $\Omega _4$ from $\Omega _3$, so there occurs no loss in $\la$.

In fact, the above proof for $u\bar{v}$ has to be reconsidered just in the cases where we have used one of modified bilinear $L^4$ estimates (Lemma~\ref{lem_RB1}--\ref{lem_DRB2}).
For other cases, modification of the argument is trivial.

For $uv$, we consider the region corresponding to $\Omega _0\cup \Omega _2\cup \Omega _3$, then the proof should be changed in treating $\Omega _{22b}$, $\Omega _{31}$ and $\Omega _{32}$.
For the first case, we use Lemma~\ref{lem_RB1} instead of Lemma~\ref{lem_RB2} and follow the above proof.
Note that we have to consider the region $\LR{k_1-(k-k_1)}\ll \LR{k_1}$ separately, where $\LR{k-k_1}\sim \LR{k_1}$ holds and we can imitate the above proof in $\Omega _{12b-0}$.
The last two cases are symmetric, where we may follow the above proof for $\Omega _{31}$, using Lemma~\ref{lem_DRB1} instead of Lemma~\ref{lem_DRB2}.

Consider $\bar{u}\bar{v}$ next.
Now, we treat the region corresponding to $\Omega _0\cup \Omega _1\cup \Omega _3$, and focus on the proof in the cases of $\Omega _{12b}$, $\Omega _{31}$ and $\Omega _{32}$.
In $\Omega _{12b}$, however, the same argument as above (with Lemma~\ref{lem_DRB2}) is applicable.
In the remaining cases are symmetric and we can apply Lemma~\ref{lem_DRB2} again to conclude the proof.

This is the end of the proof for Proposition~\ref{prop_be}.
\end{proof}


\bigskip
\section{Ill-posedness for $s<-\frac{1}{2}$}\label{sec_illp}

In this section we prove Theorem~\ref{thm_main} (II).
The scaling argument will again play a major role in the proof of ill-posedness.
Let $\la \ge 1$ be a large spatial period to be chosen later, and consider the rescaling equation \eqref{gB''} first (for a while we omit the superscript $\la$).
We observe that the nonlinear interaction of $\omega _\la ^2(u\bar{u})$ shows bad behavior below $H^{-1/2}$ in the first nonlinear iterate.

\begin{lem}\label{lem_cascade}
Let $\la \gg 1$ and $0<t_0\le 1$.
Suppose that $N\in Z^*_\la$ satisfies
\eq{cond_N}{N>0,\qquad \ttfrac{1}{2Nt_0}\ll \ttfrac{1}{\la},\qquad \ttfrac{2Nt_0}{\la}\in [\frac{\pi}{2},\frac{3\pi}{2}] \mod 2\pi .}
When $Z_\la =\T_\la$, define the $2\pi \la$-periodic function $\phi _{\la,N}$ by
\eqq{\F \phi _{\la,N}(k)=\begin{cases} 1, &k=N~\text{or}~N+\la ^{-1},\\0, &\text{otherwise}.\end{cases}}
When $Z_\la =\R$, define  $\phi _{\la,N}$ on $\R$ by
\eqq{\F \phi _{\la,N}(\xi )=\begin{cases} 1, &\xi \in [N,N+\frac{\la ^{-1}}{10}]\cup [N+\la ^{-1},N+\frac{11\la ^{-1}}{10}],\\0, &\text{otherwise}.\end{cases}}
Then, for any $s$ it holds that
\eq{claim_normofphi}{\norm{\phi _{\la,N}}{H^s(Z_\la )}\sim \la ^{-1/2}N^{s},}
\eq{claim_normofa2}{\norm{A_2(\phi _{\la,N})(t_0)}{H^s(Z_\la )}\gec \la ^{-1/2}N^{-1},}
where
\eqq{A_2(\phi )(t):=\ttfrac{i}{2}\int _0^t e^{i(t-t')\p _x^2}\omega _\la ^2\big[ e^{it'\p _x^2}\phi \cdot \bbar{e^{it'\p _x^2}\phi}\big] \,dt'.}
\end{lem}

\begin{proof}
\eqref{claim_normofphi} is easily derived from the definition.

For \eqref{claim_normofa2}, consider the case of $\Tl$ first.
It suffices to show that
\eqq{\big| \F _xA_2(\phi _{\la ,N})(t_0,\la ^{-1})\big| \gec N^{-1}.}
An explicit calculation then implies that
\eqq{\F _xA_2(\phi _{\la ,N})(t_0,\la ^{-1})=C\frac{1}{\la}\int _0^{t_0}e^{-2i\la ^{-1}Nt'}\,dt'=C'\frac{1}{N}(e^{-2i\la ^{-1}Nt_0}-1),}
where $C,C'\in \Bo{C}$ is a constant (depending on $\la$) such that $|C|,|C'|\sim 1$.
Therefore, the claim follows from \eqref{cond_N}.

Next, consider the case $Z_\la =\R$.
Again by some calculation, we see that for $\frac{9\la ^{-1}}{10}\le \xi \le \frac{11\la ^{-1}}{10}$
\eqq{\F _xA_2(\phi _{\la ,N})(t_0,\xi )&=C\ttfrac{\la ^2\xi ^2}{1+\la ^2\xi ^2}\int _\R \F \phi _{\la,N}(\xi ')\bbar{\F \phi _{\la,N}(\xi '-\xi)}\int _0^{t_0}e^{it'\{ \xi ^2-\xi '^2+(\xi '-\xi )^2\}}dt'\,d\xi '\\
&=C\ttfrac{\la ^2\xi ^2}{1+\la ^2\xi ^2}\int _{N+\la ^{-1}}^{N+\frac{11}{10}\la ^{-1}} \bbar{\F \phi _{\la,N}(\xi '-\xi)}\frac{e^{-2it_0\xi (\xi '-\xi )}-1}{\xi (\xi '-\xi )}\,d\xi '.}
In particular, for $\la ^{-1}\le \xi \le \frac{21\la ^{-1}}{20}$, 
\eqq{|\F _xA_2(\phi _{\la ,N})(t_0,\xi )|&\gec \Big| \int _{N+\xi}^{N+\frac{11}{10}\la ^{-1}} \Re \Big[ \frac{e^{-2it_0\xi (\xi '-\xi )}-1}{\xi (\xi '-\xi )}\Big] \,d\xi '\Big| \\
&\gec \frac{1}{\la ^{-1}N}\int _{N+\xi}^{N+\frac{11}{10}\la ^{-1}} \big( 1-\cos (2t_0\xi (\xi '-\xi ))\big) \,d\xi '\\
&\ge \frac{1}{\la ^{-1}N}\int _{N}^{N+\frac{1}{20}\la ^{-1}} \big( 1-\cos (2t_0\xi \xi ')\big) \,d\xi '.}
Observe that the value of $2t_0\xi \xi '$ varies over many periods while $(\xi ,\xi ')$ moves on $[\la ^{-1},\frac{21}{20}\la ^{-1}]\times [N,N+\frac{1}{20}\la ^{-1}]$, since we have assumed $\frac{t_0N}{\la}\gg 1$.
We split the interval $[\la ^{-1},\frac{21}{20}\la ^{-1}]$ into subintervals $\{ I_j\}$ of length $\frac{1}{100t_0N}\le |I_j|\le \frac{1}{50t_0N}$, then $2t_0\xi \xi '$ varies at most over $\frac{1}{10}$ periods on each interval $(\xi ,\xi ')\in I_j\times [N,N+\frac{1}{20}\la ^{-1}]$.
We only pick up intervals $I_j$ such that $\cos (2t_0\xi \xi ')\le \frac{1}{2}$ on $I_j\times [N,N+\frac{1}{20}\la ^{-1}]$, and obtain
\eqq{\norm{A_2(\phi _{\la,N})(t_0)}{H^s(\R )}&\gec \norm{\F _xA_2(\phi _{\la,N})(t_0)}{L^2([\la ^{-1},\frac{21}{20}\la ^{-1}])}\gec \frac{1}{N}\cdot \la ^{-1/2}.\qedhere}
\end{proof}

Now, fix $t_0\in (0,1]$.
For $n\in \Bo{N}$, we choose $N_n\in Z^*_\la$ satisfying \eqref{cond_N} and $N_n\to \I$ ($n\to \I$).
Let us define 
\eqq{u_{0n}:=\de N_n^{1/2}\phi _{\la ,N_n},}
where $\de>0$ is a small parameter to be chosen later, which is independent of $\la$.
Note that \eqref{claim_normofphi} implies $\tnorm{u_{0n}}{H^{-1/2}}\sim \de \la ^{-1/2}$, so we see from the proof of Theorem~\ref{thm_main}~(I) in Section~\ref{sec_def} that there exists a solution $u_n(t)$ to \eqref{gB''} on the time interval $[-1,1]$, which is unique in a closed ball of $W^{-1/2}_1$, if $\la \gg 1$ and $\de \ll 1$.

Also define
\eqs{u_{1n}(t):=e^{it\p _x^2}u_{0n},\qquad u_{2n}(t):=A_2(u_{0n})(t)=\frac{i}{2}\int _0^t e^{i(t-t')\p _x^2}\omega _\la ^2\big[ u_{1n}\bbar{u_{1n}}\big] (t')\,dt',}
and $w_n:=u_n-\big( u_{1n}+u_{2n}\big)$.
Then, $w_n$ satisfies the integral equation
\eqq{&w_n(t)=-i\la ^{-2}\int _0^te^{i(t-t')\p _x^2}\ttfrac{1}{2}\Big\{ \big( u_{1n}+u_{2n}+w_n\big) -\bbar{\big( u_{1n}+u_{2n}+w_n\big)}\Big\} (t')\,dt'\\
&+i\int _0^te^{i(t-t')\p _x^2}\ttfrac{\omega _\la ^2}{4}\Big\{ (u_{1n}+u_{2n}+w_n)^2+\bbar{(u_{1n}+u_{2n}+w_n)}^2\Big\} (t')\,dt'\\
&+i\int _0^te^{i(t-t')\p _x^2}\ttfrac{\omega _\la ^2}{2}\Big\{ u_{1n}\bbar{(u_{2n}+w_n)}+\bbar{u_{1n}}(u_{2n}+w_n)+(u_{2n}+w_n)\bbar{(u_{2n}+w_n)}\Big\} (t')\,dt'.}
On the other hand, estimates for LWP (Proposition~\ref{prop_be}, Lemma~\ref{lem_nonlinear1}) in $H^{-1/2}$ yield that
\eqs{\norm{u_{1n}}{W^{-1/2}_1}\lec \norm{u_{0n}}{H^{-1/2}}\sim \de \la ^{-1/2},\\
\norm{u_{2n}}{W^{-1/2}_1}\lec \la ^{1/2}\norm{u_{1n}}{W^{-1/2}_1}^2\lec \de ^2\la ^{-1/2}}
for any $n$.
Therefore, we can deduce from the above integral equation that
\eqq{\norm{w_n}{W^{-1/2}_1}&\lec \la ^{-2}\norm{u_{1n}+u_{2n}+w_n}{W^{-1/2}_1}+\norm{u_{1n}+u_{2n}+w_n}{W^{-1/2}_1}^2\\
&\hx +\la ^{1/2}\Big( \norm{u_{1n}}{W^{-1/2}_1}\norm{u_{2n}+w_n}{W^{-1/2}_1}+\norm{u_{2n}+w_n}{W^{-1/2}_1}^2\Big) \\
&\lec \la ^{-2}\big( \de \la ^{-1/2}+\norm{w_n}{W^{-1/2}_1}\big) +\big( \de \la ^{-1/2}+\norm{w_n}{W^{-1/2}_1}\big) ^2\\
&\hx +\la ^{1/2}\big\{ \de \la ^{-1/2}(\de ^2\la ^{-1/2}+\norm{w_n}{W^{-1/2}_1})+(\de ^2\la ^{-1/2}+\norm{w_n}{W^{-1/2}_1})^2\big\} \\
&\lec \de \la ^{-5/2}+\de ^2\la ^{-1}+\de ^3\la ^{-1/2}+(\la ^{-2}+\de )\norm{w_n}{W^{-1/2}_1}+\la ^{1/2}\norm{w_n}{W^{-1/2}_1}^2.}
If $\de$ and $\la ^{-1}$ are small enough, it follows that
\eqq{\la ^{1/2}\norm{w_n}{W^{-1/2}_1}&\lec (\de \la ^{-2}+\de ^2\la ^{-1/2}+\de ^3)+\la ^{-1/2}\big( \la ^{1/2}\norm{w_n}{W^{-1/2}_1}\big) ^2.}

It is easily verified that the above estimate also holds with the same constant when we replace $u_{0n}$ with $\th u_{0n}$, $0\le \th \le 1$ in the definition of $u_n,u_{1n},u_{2n}$ and $w_n$.
It follows from the LWP results in $H^{-1/2}$ that for each $n$,
\[ [0,1]\ni \th \mapsto \norm{w_n(\th )}{W^{-1/2}_1}\]
is continuous and $w_n(0)=0$.
Therefore, we conclude that
\eqq{\sup _n\norm{w_n}{W^{-1/2}_1}\lec \de \la ^{-5/2}+\de ^2\la ^{-1}+\de ^3\la ^{-1/2},}
and thus
\eqq{\sup _n \sup _{-1\le t\le 1}\norm{w_n(t)}{H^{-1/2}}\lec \de \la ^{-5/2}+\de ^2\la ^{-1}+\de ^3\la ^{-1/2},}
whenever $\la$ is sufficiently large.

Let $s<-\frac{1}{2}$.
Combining the above with \eqref{claim_normofphi} and \eqref{claim_normofa2}, we see that
\eqq{\norm{u_n(t_0)}{H^s}&\ge \norm{u_{2n}(t_0)}{H^s}-\norm{w_n(t_0)}{H^{-1/2}}-\norm{u_{1n}(t_0)}{H^s}\\
&\ge C^{-1}\de ^2\la ^{-1/2}-C\big( \de \la ^{-5/2}+\de ^2\la ^{-1}+\de ^3\la ^{-1/2}\big) -C\de \la ^{-1/2}N_n^{1/2+s}}
for any $n$.
Choosing $\de$ sufficiently small, and then $\la$ sufficiently large, we have
\eqq{\norm{u_n(t_0)}{H^s}&\ge C^{-1}\de ^2\la ^{-1/2}-C\de \la ^{-1/2}N_n^{1/2+s}.}
Therefore, we have shown that 
\eqq{\lim _{n\to \I}\norm{u^\la _{0n}}{H^s(Z_\la )}=0,\qquad \liminf _{n\to \I}\norm{u_n^\la (t_0)}{H^s(Z_\la )}\gec \de ^2\la ^{-1/2}.}
Note that $\de$, $\la$ are independent of the choice of $t_0$.
Rescaling back to the original equation \eqref{gB'}, we obtain
\eqq{\lim _{n\to \I}\norm{u_{0n}}{H^s(Z)}=0,\qquad \liminf _{n\to \I}\norm{u_n(\la ^{-2}t_0)}{H^s(Z)}\gec _{\de ,\la}1,}
which establishes Theorem~\ref{thm_main}~(II) with $T_0=\la ^{-2}$.



\bigskip
\appendix
\section{Proof of Proposition~\ref{prop_uniqueness}}\label{sec_MT}

Here, we give a proof of Proposition~\ref{prop_uniqueness}.
Our argument is originated in the work of Muramatu and Taoka~\cite{MT04} who considered the Cauchy problem for quadratic NLS equations in some Besov type function spaces.

We now prepare mixed $\ell ^1$-Besov spaces $B^{s,b}_{2,1}$, which is the completion of $\Sc{S}(\R \times Z)$ with respect to the norm
\eqq{\norm{v}{B^{s,b}_{2,1}}:=\textstyle\sum\limits _{j=0}^\I \textstyle\sum\limits _{k=0}^\I 2^{sj}2^{bk}\norm{p_j(\xi )p_k(\tau )\ti{v}}{L^2_{\tau ,\xi}(\R \times Z^*)},}
where the dyadic decomposition $p_j$ is defined by
\eqq{p_0:=\psi ,\qquad p_j(\xi ):=\psi (2^{-j}\xi )-\psi (2^{j-1}\xi )\quad (j=1,2,\dots )}
and $\psi$ is the bump function given in Section~\ref{sec_def}. 
We see that
\eqq{\norm{v_1(t)v_2(x)}{B^{s,b}_{2,1}}=\norm{v_1}{B^b_{2,1}(\R )}\norm{v_2}{B^s_{2,1}(Z)}.}

Then, our theorem is reduced to the following result given by Muramatu and Taoka.
\begin{thm}[\cite{MT04}]\label{thm_MT}
Let $s\in \R$ and $\frac{1}{2}\le b<1$.
Assume that $f\in \Sc{S}(\R \times Z)$ satisfies $f(0,x)\equiv 0$.
Then, we have
\eqq{\lim _{T\searrow 0}\norm{f}{B^{s,b}_{2,1,T}}=0.}
\end{thm}

We shall derive our theorem from Theorem~\ref{thm_MT}.
\begin{proof}[Proof of Proposition~\ref{prop_uniqueness}]
Assume all the condition on $\Sc{X}^s$ and $u\in \Sc{X}^s$ in Proposition~\ref{prop_uniqueness}.
First, fix an arbitrary $\e >0$ and a global-in-time extension of $u \in \Sc{X}^s_{T}$ denoted by $U$.
Then, from the density property of $\Sc{X}^s$, there exists a smooth function $V$ on $\R \times Z$ satisfying
\eq{est_uniq1}{\norm{U-V}{\Sc{X}^s}\le \e .}

We also fix such $V$.
Since $V$ is smooth,
\eqq{e^{-it\p _x^2}\big( V(t)-\psi (t)e^{it\p _x^2}V(0)\big) =e^{-it\p _x^2}V(t)-\psi (t)V(0)}
is smooth and vanishes at $t=0$.
Using the embedding $X^{s',b'}\hookrightarrow \Sc{X}^s$ with $\frac{1}{2}\le b'<1$ and the fact
\eqq{\norm{e^{it\p _x^2}v(t)}{X^{s,b}}=\norm{v}{H^b_t(H^s_x)}\lec \norm{v}{B^{s,b}_{2,1}},}
we have
\eqq{&\norm{V(t)-\psi (t)e^{it\p _x^2}V(0)}{\Sc{X}^s_{T}}\lec \norm{e^{-it\p _x^2}V(t)-\psi (t)V(0)}{B^{s',b'}_{2,1,T}}.}
Applying Theorem~\ref{thm_MT}, we see that
\eq{est_uniq2}{\lim _{T\searrow 0}\norm{V(t)-\psi (t)e^{it\p _x^2}V(0)}{\Sc{X}^s_{T}}=0.}

We next use the embedding $X^{s,b}\hookrightarrow \Sc{X}^s$ with $b>\frac{1}{2}$ to have
\eqq{&\norm{\psi (t)e^{it\p _x^2}V(0)}{\Sc{X}^s}\lec \norm{V(0)}{H^s(Z)}=\norm{(U-V)(0)}{H^s(Z)}\le \sup _{t}\norm{(U-V)(t)}{H^s(Z)}.}
Here, we have used $U(0)=0$.
From the embedding $\Sc{X}^s\hookrightarrow C_t(\R ;H^s(Z))$ and \eqref{est_uniq1}, we conclude that
\eq{est_uniq3}{\norm{\psi (t)e^{it\p _x^2}V(0)}{\Sc{X}^s}\lec \e .}

Combining \eqref{est_uniq1}--\eqref{est_uniq3}, we obtain
\eqq{\limsup _{T\searrow 0}\norm{u}{\Sc{X}^s_{T}}\lec \e .}
Since $\e$ is arbitrary, the claim follows.
\end{proof}

In the rest of this section, we shall again describe the proof of Theorem~\ref{thm_MT}.
To do this, the following difference norm of $B^{s,b}_{2,1}$ will be useful.
\begin{lem}\label{lem_equivnorm}
Let $s\in \R$ and $0<b<1$.
Then, we have
\eqq{\norm{f}{B^{s,b}_{2,1}}\sim \sum _{j=0}^\I 2^{sj}\norm{f_j}{L^2(\R \times Z)}+ \sum _{j=0}^\I 2^{sj}\int _\R |r|^{-b}\norm{f_j(t+r,x)-f_j(t,x)}{L^2(\R \times Z)}\ttfrac{dr}{|r|},}
where $f_j(t,x)=P_jf(t,x):=\F^{-1}_\xi p_j(\xi )\F _xf(t,\xi )$.
\end{lem}

Note that the usual Besov norm has a similar representation
\eq{equivnorm_1d}{\norm{f}{B^{b}_{2,1}(\R )}\sim \norm{f}{L^2(\R )}+\int _\R |r|^{-b}\norm{f(\cdot +r)-f(\cdot )}{L^2(\R )}\ttfrac{dr}{|r|}}
for $0<b<1$.
From definition of the $B^{s,b}_{2,1}$ norm, it suffices to show
\eqq{\sum _{k=0}^\I 2^{bk}\norm{f_{jk}}{L^2(\R \times Z)}\sim \norm{f_j}{L^2(\R \times Z)}+ \int _\R |r|^{-b}\norm{f_j(t+r,x)-f_j(t,x)}{L^2(\R \times Z)}\ttfrac{dr}{|r|}}
for each $j$, where $f_{j,k}:=\F _{t,x}^{-1}p_j(\xi )p_k(\tau )\ti{f}$.
This equivalence can be verified in the same way as the proof for \eqref{equivnorm_1d}.
Thus, we omit the proof and refer to \cite{MT04}, Theorem~8.1.

Using Lemma~\ref{lem_equivnorm}, we can verify a similar representation for the restricted norm.
\begin{lem}\label{lem_final}
Let $s\in \R$ and $\frac{1}{2}\le b<1$.
Define
\eqs{G_T(f;\rho ):=\sum _{j=0}^\I 2^{sj}\norm{f_j(t+\rho ,x)-f_j(t,x)}{L^2(I_{T,\rho}\times Z)},\\
I_{T,\rho}:=[-T,T]\cap [-T-\rho ,T-\rho ].}
for $T>0$, $f\in \Sc{S}(\R \times Z)$, and $\rho\in \R$.
Then, we have
\begin{align}
\int _{-2T}^{2T}|\rho |^{-b}G_T(f;\rho )\ttfrac{d\rho}{|\rho|}&\lec \norm{f}{B^{s,b}_{2,1,T}}\label{est_final1}\\
&\lec \int _{-2T}^{2T}|\rho |^{-b}G_T(f;\rho )\ttfrac{d\rho}{|\rho|}+\norm{f(0,\cdot )}{B^s_{2,1}}.\label{est_final2}
\end{align}
for any $f\in \Sc{S}(\R \times Z)$ and $T\in (0,1]$.
\end{lem}

Assume Lemma~\ref{lem_final} for now, then it suffices for the proof of Theorem~\ref{thm_MT} to show that
\eqq{\lim_{T\searrow 0}\int _{-2T}^{2T}|\rho |^{-b}G_T(f;\rho )\ttfrac{d\rho}{|\rho|}=0}
for $f\in \Sc{S}(\R \times Z)$.
However, we observe that $G_T(f;\rho )\le G_1(f;\rho )$ for any $f$ and $\rho$.
\eqref{est_final1} shows that $\rho ^{-b-1}G_1(f;\rho )$ is integrable on $\rho \in(-2,2)$, so we conclude that
\eqq{\int _{-2T}^{2T}|\rho |^{-b}G_T(f;\rho )\ttfrac{d\rho}{|\rho|}\le \int _{-2T}^{2T} \rho ^{-b}G_1(f;\rho )\ttfrac{d\rho}{\rho}\to 0\qquad (T\to 0),}
as desired.

It remains to prove Lemma~\ref{lem_final}.
The most important problem is what function in $B^{s,b}_{2,1}$ we should choose as a global-in-time extension of $f\big| _{t\in [-T,T]}\in B^{s,b}_{2,1,T}$.
In this point of view, the following representation lemma plays a great role.
\begin{lem}[\cite{M74}]\label{lem_rep}
Let $\omega \in C^\I (\R )$ be a non-negative function satisfying
\eqq{\supp{\omega} \subset [-1,1],\qquad \int _\R \omega (z)\,dz=1.}
Define the following functions:
\eqq{K(t,z):=z\omega (z-t),\qquad M(t,z):=\p _zK(t,z),}
and also
\begin{align*}
K_1(t,z)&:=K(t,z),& K_2(t,z)&:=K(t,z),& K_3(t,z)&:=M(t,z),\\
L_1(t,z)&:=-\p _tM(t,z),& L_2(t,z)&:=\p _zM(t,z),& L_3(t,z)&:=M(t,z).
\end{align*}
Let $T>0$ and $f\in \Sc{S}(\R)$.
Then, the following identities hold for any $t\in [-T,T]$:
\begin{align}
f(t)&=\int _0^T\int _\R \ttfrac{1}{\mu}M(\ttfrac{t}{T},\ttfrac{t-t_1}{\mu})f(t_1)\,dt_1\,\ttfrac{d\mu}{\mu}+f_0\label{id_uniq1}\\
&=\sum _{i=1}^3\int _0^T\int _\R \ttfrac{1}{\mu}K_i(\ttfrac{t}{T},\ttfrac{t-t_1}{\mu})f_i(\mu ,t_1)\,dt_1\,\ttfrac{d\mu}{\mu}+f_0,\label{id_uniq2}
\end{align}
where
\eqq{f_0&:=\int _\R \ttfrac{1}{T}\omega (\ttfrac{-t_1}{T})f(t_1,x)\,dt_1,\\
f_1(\mu ,t_1)&:=\int _\mu ^T\int _\R \ttfrac{1}{\nu}L_1(\ttfrac{t_1}{T},\ttfrac{t_1-t_2}{\nu})f(t_2)\,dt_2\,\ttfrac{-\mu}{T}\ttfrac{d\nu}{\nu},\\
f_2(\mu ,t_1)&:=\int _\mu ^T\int _\R \ttfrac{1}{\nu}L_2(\ttfrac{t_1}{T},\ttfrac{t_1-t_2}{\nu})f(t_2)\,dt_2\,\ttfrac{\mu}{\nu}\ttfrac{d\nu}{\nu},\\
f_3(\mu ,t_1)&:=\int _0^\mu \int _\R \ttfrac{1}{\nu}L_3(\ttfrac{t_1}{T},\ttfrac{t_1-t_2}{\nu})f(t_2)\,dt_2\,\ttfrac{d\nu}{\nu}.}
\end{lem}

\begin{proof}
For $(\mu ,t)\in (0,\I )\times \R$, define
\eqq{h(\mu ,t):=\int _\R \ttfrac{1}{\mu}\omega (\ttfrac{t-t_1}{\mu}-\ttfrac{t}{T})f(t_1)\,dt_1,\quad H(\mu ,t):=\int _\R \ttfrac{1}{\mu}M(\ttfrac{t}{T},\ttfrac{t-t_1}{\mu})f(t_1)\,dt_1.}
Note that $h(T,t)=f_0$.
Since $\int_\R \tfrac{1}{\mu}\omega (\tfrac{t-t_1}{\mu}-\tfrac{t}{T})\,dt_1=1$, we have
\eqq{h(\mu ,t)-f(t)&=\int _\R \ttfrac{1}{\mu}\omega (\ttfrac{t-t_1}{\mu}-\ttfrac{t}{T})\{ f(t_1)-f(t)\} \,dt_1\\
&=-\int _\R \omega (z-\ttfrac{t}{T})\{ f(t-\mu z)-f(t)\} \,dz.}
We observe that $\supp{\omega (\cdot -\frac{t}{T})}\subset [-2,2]$ for any $t\in [-T,T]$, so the Lebesgue convergence theorem implies that $h(\mu ,t)\to f(t)$ as $\mu \to 0$ uniformly in $t\in [-T,T]$.

On the other hand, observe that $\p _\mu \big(\frac{1}{\mu}\omega (\frac{t-t_1}{\mu}-\frac{t}{T})\big) =-\frac{1}{\mu^2}M(\frac{t}{T},\frac{t-t_1}{\mu})$, then
\eqq{\p _\mu h(\mu ,t)=-\ttfrac{1}{\mu}H(\mu ,t).}
Thus, we have
\eqq{f_0-f(t)=-\int _0^TH(\mu ,t)\,\ttfrac{d\mu}{\mu},}
which shows \eqref{id_uniq1}.
We substitute \eqref{id_uniq1} into itself to have
\eqq{f(t)=&f_0+f_0\int _0^T\int _\R \ttfrac{1}{\mu}M(\ttfrac{t}{T},\ttfrac{t-t_1}{\mu})\,dt_1\,\ttfrac{d\mu}{\mu}\\
&+\int _0^T\int _\R \ttfrac{1}{\mu}M(\ttfrac{t}{T},\ttfrac{t-t_1}{\mu})\int _\mu ^T\int _\R \ttfrac{1}{\nu}M(\ttfrac{t_1}{T},\ttfrac{t_1-t_2}{\nu})f(t_2)\,dt_2\,\ttfrac{d\nu}{\nu}\,dt_1\,\ttfrac{d\mu}{\mu}\\
&+\int _0^T\int _\R \ttfrac{1}{\mu}M(\ttfrac{t}{T},\ttfrac{t-t_1}{\mu})\int _0^\mu \int _\R \ttfrac{1}{\nu}M(\ttfrac{t_1}{T},\ttfrac{t_1-t_2}{\nu})f(t_2)\,dt_2\,\ttfrac{d\nu}{\nu}\,dt_1\,\ttfrac{d\mu}{\mu}.}
The second term vanishes because $\int _\R M(t,z)\,dz=0$.
We next observe that
\eqq{\p _{t_1}K(\ttfrac{t}{T},\ttfrac{t-t_1}{\mu})=-\ttfrac{1}{\mu}M(\ttfrac{t}{T},\ttfrac{t-t_1}{\mu}).}
Then, after an integration by parts with respect to $t_1$, the third term becomes
\eqq{&\int _0^T\int _\R K(\ttfrac{t}{T},\ttfrac{t-t_1}{\mu})\p _{t_1}\bigg[ \int _\mu ^T\int _\R \ttfrac{1}{\nu}M(\ttfrac{t_1}{T},\ttfrac{t_1-t_2}{\nu})f(t_2)\,dt_2\,\ttfrac{d\nu}{\nu}\bigg] \,dt_1\,\ttfrac{d\mu}{\mu}\\
&\hx =\int _0^T\int _\R \ttfrac{1}{\mu}K(\ttfrac{t}{T},\ttfrac{t-t_1}{\mu})\int _\mu ^T\int _\R \ttfrac{\mu}{\nu}\Big[ \ttfrac{-1}{T}L_1+\ttfrac{1}{\nu}L_2\Big] (\ttfrac{t_1}{T},\ttfrac{t_1-t_2}{\nu})f(t_2)\,dt_2\,\ttfrac{d\nu}{\nu}\,dt_1\,\ttfrac{d\mu}{\mu}.}
Collecting the above, we obtain \eqref{id_uniq2}.
\end{proof}

We prepare one more lemma.
\begin{lem}\label{lem_lem_rep}
Let $0<T\le 1$.

(i) Assume that $K_*:\R \times \R \to \R$ is a smooth function satisfying
\eqq{(t,z)\in \supp{K_*}\qquad \Rightarrow |t-z|\le 1.}
For $(\mu ,t,x)\in (0,T)\times \R \times Z$ and $g:(0,T)\to \Sc{S}(\R \times Z)$, define
\eqq{H_K(g;\mu ,t,x):=\int _\R \ttfrac{1}{\mu}\psi (\ttfrac{t}{T})K_*(\ttfrac{t}{T},\ttfrac{t-t_1}{\mu})g(\mu ,t_1,x)\,dt_1,}
where $\psi$ is the same bump function as before.
Then, we have
\begin{gather}
\label{est_lem_rep1} \norm{H_K(g;\mu )}{L^2(\R \times Z)}\lec \norm{g(\mu )}{L^2(\R \times Z)},\\
\label{est_lem_rep2} \norm{H_K(g;\mu ,t+r,x)-H_K(g; \mu ,t,x)}{L^2_{t,x}(\R \times Z)}\lec \min \shugo{1,\,\ttfrac{|r|}{\mu}}\norm{g(\mu )}{L^2(\R \times Z)}
\end{gather}
for any $0<\mu <T$ and $r\in \R$.
The implicit constants depend only on $K_*$.

(ii) Assume that $L_*:\R \times \R \to \R$ is a smooth function satisfying $\int _\R L_*(t,z)\,dz\equiv 0$ and
\eqq{(t,z)\in \supp{L_*}\qquad \Rightarrow |t-z|\le 1.}
For $(\nu ,t_1,x)\in (0,T)\times \R \times Z$ and $g\in \Sc{S}(\R \times Z)$, define
\eqq{H_L(g;\nu ,t_1,x):=\int _\R \ttfrac{1}{\nu}\psi (\ttfrac{t_1}{5T})L_*(\ttfrac{t_1}{T},\ttfrac{t_1-t_2}{\nu})g(t_2,x)\,dt_2.}
Then, we have
\eq{est_lem_rep3}{\norm{H_L(g;\nu )}{L^2(\R \times Z)}\lec \int _\R \ttfrac{1}{\nu}\chf{[-11\nu ,11\nu ]}(\rho )\norm{g(t_1+\rho ,x)-g(t_1,x)}{L^2_{t_1,x}(\R \times Z)}\,d\rho}
for any $0<\nu <T$.
The implicit constant depends only on $L_*$.
\end{lem}

Note that the functions $K_1,K_2,K_3$ defined in Lemma~\ref{lem_rep} satisfy the condition for $K_*$.
Also, $L_1$, $L_2$, and $L_3$ all meet the condition for $L_*$.

For the proof of Lemma~\ref{lem_lem_rep}, we recall the following.
\begin{lem}\label{lem_lem2_rep}
Suppose that a smooth function $k$ on $\R \times \R$ satisfies
\eq{cond_lem_rep}{\sup _t\int _\R |k(t,t_1)|\,dt_1\le C_0,\qquad \sup _{t_1}\int _\R |k(t,t_1)|\,dt\le C_0}
for some $C_0>0$.
Then, for any $f\in \Sc{S}(\R \times Z)$, we have
\eqq{\norm{\int _\R k(t,t_1)f(t_1,x)\,dt_1}{L^2_{t,x}(\R \times Z)}\le C_0\norm{f}{L^2(\R \times Z)}.}
\end{lem}

\begin{proof}
We use the Minkowski and the Cauchy-Schwarz inequalities to obtain
\eqq{&\norm{\int _\R k(t,t_1)f(t_1,x)\,dt_1}{L^2_{t,x}(\R \times Z)}^2\\
&\hx \le \int _\R \bigg( \int _\R |k(t,t_1)|\norm{f(t_1)}{L^2_x(Z)}\,dt_1\bigg) ^2\,dt\\
&\hx \le \int _\R \bigg( \int _\R |k(t,t_1)|\,dt_1\bigg) \bigg( \int _\R |k(t,t_1)|\norm{f(t_1)}{L^2_x(Z)}^2\,dt_1\bigg) \,dt\\
&\hx \le \sup _t\int _\R |k(t,t_1)|\,dt_1\cdot \int _\R \norm{f(t_1)}{L^2_x(Z)}^2\bigg( \int _\R |k(t,t_1)|\,dt\bigg) \,dt_1\\
&\hx \le C_0^2 \norm{f}{L^2(\R \times Z)}^2.\qedhere}
\end{proof}

\begin{proof}[Proof of Lemma~\ref{lem_lem_rep}]
(i) Put $k(t,t_1)=\frac{1}{\mu}\psi (\frac{t}{T})K_*(\frac{t}{T},\frac{t-t_1}{\mu})$.
From changes of variables, we see that
\eqs{\int _\R |k(t,t_1)|\,dt_1\le \int _\R |K_*(\ttfrac{t}{T},z)|\,dz\lec 1,\\
\int _\R |k(t,t_1)|\,dt\le \int _\R \psi (\ttfrac{t_1+\mu z}{T})|K_*(\ttfrac{t_1+\mu z}{T},z)|\,dz\lec \int _\R \chf{[-3,3]}(z)\,dz\lec 1.}
(Note that $|\frac{t_1+\mu z}{T}|\le 2$ and $|\frac{t_1+\mu z}{T}-z|\le 1$ imply $|z|\le 3$.)
Then, \eqref{est_lem_rep1} follows from Lemma~\ref{lem_lem2_rep}.

Also, for \eqref{est_lem_rep2}, it suffices to verify that
\eqq{k(t,t_1)=\ttfrac{1}{\mu}\psi (\ttfrac{t+r}{T})K_*(\ttfrac{t+r}{T},\ttfrac{t+r-t_1}{\mu})-\ttfrac{1}{\mu}\psi (\ttfrac{t}{T})K_*(\ttfrac{t}{T},\ttfrac{t-t_1}{\mu})}
satisfies \eqref{cond_lem_rep} with $C_0\lec \min \shugo{1,\,\frac{|r|}{\mu}}$.
The triangle inequality and the above argument imply that $C_0\lec 1$ for any $r\in \R$.
On the other hand, an application of the mean value theorem followed by change of variables shows that $C_0\lec \frac{|r|}{\mu}$ if $|r|\le \mu$, as desired.

(ii) We first see that
\eqq{H_L(g;\nu ,t_1,x)=\int _\R \ttfrac{1}{\nu}\psi (\ttfrac{t_1}{5T})L_*(\ttfrac{t_1}{T},\ttfrac{-\rho}{\nu})g(t_1+\rho ,x)\,d\rho .}
The term $0\equiv -\int _\R \frac{1}{\nu}\psi (\frac{t_1}{5T})L_*(\frac{t_1}{T},\frac{-\rho}{\nu})g(t_1,x)\,d\rho$ is added and we have
\eqq{H_L(g;\nu ,t_1,x)=\int _\R \ttfrac{1}{\nu}\psi (\ttfrac{t_1}{5T})L_*(\ttfrac{t_1}{T},\ttfrac{-\rho}{\nu})\big\{ g(t_1+\rho ,x)-g(t_1,x)\big\} \,d\rho .}
Then, it is sufficient to show that
\eq{cond_rep1}{\big| \ttfrac{t_1}{5T}\big| \le 2,\qquad \big| \ttfrac{t_1}{T}+\ttfrac{\rho}{\nu}\big| \le 1}
implies
\eq{cond_rep2}{-11\nu \le \rho \le 11\nu, \qquad t_1\in I_{T,\rho}=[-T,T]\cap [-T-\rho ,T-\rho ].}

It follows from \eqref{cond_rep1} that
\eqq{\big| \ttfrac{\rho}{\nu}\big| \le \big| \ttfrac{t_1}{T}+\ttfrac{\rho}{\nu}\big| +\big| \ttfrac{t_1}{T}\big| \le 11,}
and we have the first one in \eqref{cond_rep2}.
It remains to prove $\big| \frac{t_1+\rho}{T}\big| \le 1$, for which we divide the analysis corresponding to the sign of $t_1$ and $\rho$.
If $t_1,\rho \ge 0$, then we have
\eqq{0\le \ttfrac{t_1+\rho}{T}\le \ttfrac{t_1}{T}+\ttfrac{\rho}{\nu}\le 1.}
If $t_1\le 0 \le \rho$, then
\eqq{-1\le \ttfrac{t_1}{T}\le \ttfrac{t_1+\rho}{T}\le \ttfrac{t_1}{T}+\ttfrac{\rho}{\nu}\le 1.}
The remaining cases are reduced to the above two cases.
\end{proof}

Finally, we shall prove Lemma~\ref{lem_final} and conclude this section.
\begin{proof}[Proof of Lemma~\ref{lem_final}]
Let $F$ be any extension of $f\big| _{t\in [-T,T]}\in B^{s,b}_{2,1,T}$.
Then, we use Lemma~\ref{lem_equivnorm} to see that
\eqq{\int _{-2T}^{2T}|\rho |^{-b}G_T(f;\rho )\ttfrac{d\rho}{|\rho|}\le \int _\R |\rho |^{-b}G_\I (F;\rho )\ttfrac{d\rho}{|\rho|}\lec \norm{F}{B^{s,b}_{2,1}}.}
Taking infimum over $F$, we obtain \eqref{est_final1}.

For \eqref{est_final2}, we first estimate $\tnorm{f(0,x)}{B^{s,b}_{2,1,T}}$.
Choosing $\psi (t)f(0,x)$ as an extension of $\chf{[-T,T]}(t)f(0,x)\in B^{s,b}_{2,1,T}$, we have
\eqq{\norm{f(0,x)}{B^{s,b}_{2,1,T}}\le \norm{\psi}{B^b_{2,1}(\R )}\norm{f(0,\cdot )}{B^s_{2,1}(Z)}\lec \norm{f(0,\cdot )}{B^s_{2,1}(Z)}.}

Therefore, it remains to show
\eqq{\norm{f(t,x)-f(0,x)}{B^{s,b}_{2,1,T}}\lec \int _{-2T}^{2T}|\rho |^{-b}G_T(f;\rho )\ttfrac{d\rho}{|\rho|}.}
We note that the integral in the right-hand side with respect to $\rho$ can be replaced with $\int _\R$, because $I_{T,\rho }=\emptyset$ if $|\rho |>2T$.

From the representation formula \eqref{id_uniq2}, we have the following identity for $(t,x)\in [-T,T]\times Z$:
\eqq{f(t,x)-f(0,x)=&\sum _{i=1}^3\int _0^T\int _\R \ttfrac{1}{\mu}K_i(\ttfrac{t}{T},\ttfrac{t-t_1}{\mu})f_i(\mu ,t_1,x)\,dt_1\,\ttfrac{d\mu}{\mu}\\
&-\sum _{i=1}^3\int _0^T\int _\R \ttfrac{1}{\mu}K_i(0,\ttfrac{-t_1}{\mu})f_i(\mu ,t_1,x)\,dt_1\,\ttfrac{d\mu}{\mu},}
\eqq{
f_1(\mu ,t_1,x)&:=\int _\mu ^T\int _\R \ttfrac{1}{\nu}L_1(\ttfrac{t_1}{T},\ttfrac{t_1-t_2}{\nu})f(t_2,x)\,dt_2\,\ttfrac{-\mu}{T}\ttfrac{d\nu}{\nu},\\
f_2(\mu ,t_1,x)&:=\int _\mu ^T\int _\R \ttfrac{1}{\nu}L_2(\ttfrac{t_1}{T},\ttfrac{t_1-t_2}{\nu})f(t_2,x)\,dt_2\,\ttfrac{\mu}{\nu}\ttfrac{d\nu}{\nu},\\
f_3(\mu ,t_1,x)&:=\int _0^\mu \int _\R \ttfrac{1}{\nu}L_3(\ttfrac{t_1}{T},\ttfrac{t_1-t_2}{\nu})f(t_2,x)\,dt_2\,\ttfrac{d\nu}{\nu}.}
Then, we use the following function as an extension of $f(t,x)-f(0,x)\big| _{t\in[-T,T]}\in B^{s,b}_{2,1,T}$:
\eq{id_uniq11}{&\sum _{i=1}^3\int _0^T\int _\R \ttfrac{1}{\mu}\psi (\ttfrac{t}{T})K_i(\ttfrac{t}{T},\ttfrac{t-t_1}{\mu})\psi (\ttfrac{t_1}{5T})f_i(\mu ,t_1,x)\,dt_1\,\ttfrac{d\mu}{\mu}\\
&-\sum _{i=1}^3\int _0^T\int _\R \ttfrac{1}{\mu}\psi(t)K_i(0,\ttfrac{-t_1}{\mu})\psi (\ttfrac{t_1}{5T})f_i(\mu ,t_1,x)\,dt_1\,\ttfrac{d\mu}{\mu}.}
We see from the support property of $\psi$ and $K_i$ that the above function is actually equal to $f(t)-f(0)$ on $[-T,T]$.

We now estimate the first three terms in \eqref{id_uniq11}, which is simply written as
\eqs{\sum _{i=1}^3\int _0^TF_i(\mu ,t,x)\,\ttfrac{d\mu}{\mu},\quad F_i(\mu ,t,x):=\int _\R \ttfrac{1}{\mu}\psi (\ttfrac{t}{T})K_i(\ttfrac{t}{T},\ttfrac{t-t_1}{\mu})\psi (\ttfrac{t_1}{5T})f_i(\mu ,t_1,x)\,dt_1.}

Applying Lemma~\ref{lem_equivnorm}, we reduce the estimate of $\tnorm{F_i(\mu )}{B^{s,b}_{2,1}}$ to that of
\begin{gather}
\sum _{j=0}^\I 2^{sj}\norm{P_jF_i(\mu )}{L^2(\R \times Z)},\label{quan1}\\
\sum _{j=0}^\I 2^{sj}\int _\R |r|^{-b}\norm{P_jF_i(\mu ,t+r,x)-P_jF_i(\mu ,t,x)}{L^2_{t,x}(\R \times Z)}\,\ttfrac{dr}{|r|}.\label{quan2}
\end{gather}

For \eqref{quan1}, we apply Lemma~\ref{lem_lem_rep} (i) to bound $\tnorm{P_jF_i(\mu )}{L^2(\R \times Z)}$ by
\eqq{\norm{\psi (\ttfrac{t_1}{5T})P_jf_i(\mu ,t_1,x)}{L^2_{t_1,x}(\R \times Z)},}
which is then estimated with Lemma~\ref{lem_lem_rep} (ii) by
\eqs{\int _\mu ^T\int _\R \ttfrac{1}{\nu} \chf{[-11\nu,11\nu]}(\rho )\norm{P_jf(t_1+\rho )-P_jf(t_1,x)}{L^2_{t_1,x}(I_{T,\rho }\times Z)}d\rho \,\ttfrac{\mu}{T}\ttfrac{d\nu}{\nu},\\
\int _\mu ^T\int _\R \ttfrac{1}{\nu} \chf{[-11\nu,11\nu]}(\rho )\norm{P_jf(t_1+\rho )-P_jf(t_1,x)}{L^2_{t_1,x}(I_{T,\rho }\times Z)}d\rho \,\ttfrac{\mu}{\nu}\ttfrac{d\nu}{\nu},\\
\int _\mu ^T\int _\R \ttfrac{1}{\nu} \chf{[-11\nu,11\nu]}(\rho )\norm{P_jf(t_1+\rho )-P_jf(t_1,x)}{L^2_{t_1,x}(I_{T,\rho }\times Z)}d\rho \,\ttfrac{d\nu}{\nu}.}
Therefore, \eqref{quan1} is bounded by
\eqs{\int _\mu ^T\int _\R \ttfrac{1}{\nu} \chf{[-11\nu,11\nu]}(\rho )G_T(f;\rho )\ttfrac{d\rho}{|\rho |} \,|\rho| \ttfrac{\mu}{T}\ttfrac{d\nu}{\nu},\\
\int _\mu ^T\int _\R \ttfrac{1}{\nu} \chf{[-11\nu,11\nu]}(\rho )G_T(f;\rho )\ttfrac{d\rho}{|\rho |} \,|\rho| \ttfrac{\mu}{\nu}\ttfrac{d\nu}{\nu},\\
\int _\mu ^T\int _\R \ttfrac{1}{\nu} \chf{[-11\nu,11\nu]}(\rho )G_T(f;\rho )\ttfrac{d\rho}{|\rho |} \,|\rho| \ttfrac{d\nu}{\nu}.}
We next compute the integral with respect to $\nu$.
For the first and the second one, we have
\eqq{&\int _\mu ^T\ttfrac{1}{\nu} \chf{[-11\nu,11\nu]}(\rho )|\rho |\ttfrac{\mu}{T}\ttfrac{d\nu}{\nu}\le \int _\mu ^T\ttfrac{1}{\nu} \chf{[-11\nu,11\nu]}(\rho )|\rho |\ttfrac{\mu}{\nu}\ttfrac{d\nu}{\nu}
\lec \min \shugo{\ttfrac{\mu}{|\rho |},\,\ttfrac{|\rho |}{\mu}}.}
Similarly, we obtain the bound $\chf{[-11,11]}(\frac{|\rho |}{\mu})$ for the last one.
Collecting these results, we have
\eqq{\eqref{quan1}\lec \int _\R \Big( \min \shugo{\ttfrac{\mu}{|\rho |},\,\ttfrac{|\rho |}{\mu}}+\chf{[-11,11]}(\ttfrac{|\rho |}{\mu})\Big) G_T(f;\rho )\ttfrac{d\rho}{|\rho |}.}
In the same manner, Lemma~\ref{lem_lem_rep} implies that
\eqq{\eqref{quan2}&\lec \int _\R \Big( \min \shugo{\ttfrac{\mu}{|\rho |},\,\ttfrac{|\rho |}{\mu}}+\chf{[-11,11]}(\ttfrac{|\rho |}{\mu})\Big) \int _\R |r|^{-b}\min \shugo{1,\,\ttfrac{|r|}{\mu}}\,\ttfrac{dr}{|r|}\,G_T(f;\rho )\ttfrac{d\rho}{|\rho |}\\
&\lec \mu ^{-b}\int _\R \Big( \min \shugo{\ttfrac{\mu}{|\rho |},\,\ttfrac{|\rho |}{\mu}}+\chf{[-11,11]}(\ttfrac{|\rho |}{\mu})\Big) G_T(f;\rho )\ttfrac{d\rho}{|\rho |}.}
Then, we calculate the integral with respect to $\mu$,
\eqq{&\int _0^T\big( 1+\mu ^{-b}\big) \Big( \min \shugo{\ttfrac{\mu}{|\rho |},\,\ttfrac{|\rho |}{\mu}}+\chf{[-11,11]}(\ttfrac{|\rho |}{\mu})\Big) \,\ttfrac{d\mu}{\mu}\\
&\hx \lec \int _0^{|\rho |}\mu ^{-b}\,\ttfrac{d\mu}{|\rho |} +\int _{|\rho |}^\I \mu ^{-b-2}|\rho |\,d\mu +\int _{|\rho |/11}^\I \mu ^{-b-1}d\mu ~\lec ~|\rho |^{-b},}
and complete the estimate for the first three terms in \eqref{id_uniq11}.

The second three terms in \eqref{id_uniq11} are much easier to treat.
Let us estimate the term including $f_1$, for instance.
First, we have
\eqq{&\norm{\int _0^T\int _\R \ttfrac{1}{\mu}\psi (t)K_1(0,\ttfrac{-t_1}{\mu})\psi (\ttfrac{t_1}{5T})f_1(\mu ,t_1,x)\,dt_1\,\ttfrac{d\mu}{\mu}}{B^{s,b}_{2,1}}\\
&\hx =\norm{\psi}{B^b_{2,1}(\R )}\norm{\int _0^T\int _\R \ttfrac{1}{\mu}K_1(0,\ttfrac{-t_1}{\mu})\psi (\ttfrac{t_1}{5T})f_1(\mu ,t_1,x)\,dt_1\,\ttfrac{d\mu}{\mu}}{B^s_{2,1}(Z)}\\
&\hx \lec \sum _{j=0}^\I 2^{sj}\int _0^T\int _\R \ttfrac{1}{\mu}\big| K_1(0,\ttfrac{-t_1}{\mu})\big| \norm{\psi (\ttfrac{t_1}{5T})P_jf_1(\mu ,t_1,x)}{L^2_{x}(Z)}\,dt_1\,\ttfrac{d\mu}{\mu}.}
We apply the Cauchy-Schwarz inequality in $t_1$ to bound it by
\eqq{\sum _{j=0}^\I 2^{sj}\int _0^T\ttfrac{1}{\mu}\Big( \int _\R \big| K_1(0,\ttfrac{-t_1}{\mu})\big| ^2\,dt_1\Big) ^{1/2}\norm{\psi (\ttfrac{t_1}{5T})P_jf_1(\mu ,t_1,x)}{L^2_{t_1,x}(\R \times Z)}\,\ttfrac{d\mu}{\mu}.}
On the other hand, a simple calculation shows that
\eqq{\ttfrac{1}{\mu}\Big( \int _\R \big| K_1(0,\ttfrac{-t_1}{\mu})\big| ^2\,dt_1\Big) ^{1/2}~\lec \mu ^{-1/2}\lec \mu ^{-b}.}
Thus, we obtain exactly the same quantity as that we have estimated above.
The other two terms are treated in the same way, and the proof is completed.
\end{proof}

\bigskip
\section*{Acknowledgments}

The author would like to express his great appreciation to Prof. Kotaro Tsugawa for a number of precious suggestions.
This work was partially supported by Grant-in-Aid for JSPS Fellows 08J02196.



\bigskip
\begin{thebibliography}{00}
\bibitem{BT06} I. Bejenaru and T. Tao, \textit{Sharp well-posedness and ill-posedness results for a quadratic non-linear Schr\"odinger equation}, J. Funct. Anal. \textbf{233} (2006), no. 1, 228-259.
Latest version is available in \texttt{arXiv:math/0508210v4}.
\bibitem{BS88} J.L. Bona and R.L. Sachs, \textit{Global existence of smooth solutions and stability of solitary waves for a generalized Boussinesq equation}, Comm. Math. Phys. \textbf{118} (1988), no. 1, 15--29.
\bibitem{B93-1} J. Bourgain, \textit{Fourier transform restriction phenomena for certain lattice subsets and applications to nonlinear evolution equations, I, Schr\"odinger equations}, Geom. Funct. Anal. \textbf{3} (1993), no. 2, 107--156.
\bibitem{DTT82} P. Deift, C. Tomei, and E. Trubowitz, \textit{Inverse scattering and the Boussinesq equation}, Comm. Pure Appl. Math. \textbf{35} (1982), no. 5, 567--628.
\bibitem{FLS87} F. Falk, E.W. Laedke, and K.H. Spatschek, \textit{Stability of solitary-wave pulses in shape-memory alloys}, Phys. Review B \textbf{36} (1987), no. 6, 3031--3041.
\bibitem{FG96} Y.-F. Fang and M.G. Grillakis, \textit{Existence and uniqueness for Boussinesq type equations on a circle}, Comm. Partial Differential Equations \textbf{21} (1996), no. 7-8, 1253--1277.
\bibitem{F09} L.G. Farah, \textit{Local solutions in Sobolev spaces with negative indices for the ``good'' Boussinesq equation}, Comm. Partial Differential Equations \textbf{34} (2009), no. 1-3, 52--73.
\bibitem{FS10} L.G. Farah and M. Scialom, \textit{On the periodic ``good'' Boussinesq equation}, Proc. Amer. Math. Soc. \textbf{138} (2010), no. 3, 953--964.
\bibitem{G00p} A. Gr\"unrock, \textit{Some local wellposedness results for nonlinear Schr\"odinger equations below $L^2$}, preprint (2000). \texttt{arXiv:math/0011157}
\bibitem{KL78} V.K. Kalantarov and O.A. Ladyzhenskaya, \textit{The occurrence of collapse for quasilinear equations of parabolic and hyperbolic types}, J. Soviet Math. \textbf{10} (1978), no. 1, 53--70.
\bibitem{KPV96-NLS} C.E. Kenig, G. Ponce, and L. Vega, \textit{Quadratic forms for the 1-D semilinear Schr\"odinger equation}, Trans. Amer. Math. Soc. \textbf{348} (1996), no. 8, 3323--3353.
\bibitem{K09-NLS} N. Kishimoto, \textit{Low-regularity bilinear estimates for a quadratic nonlinear Schr\"odinger equation}, J. Differential Equations \textbf{247} (2009), no. 5, 1397--1439.
\bibitem{K09-KdV} N. Kishimoto, \textit{Well-posedness of the Cauchy problem for the Korteweg-de Vries equation at the critical regularity}, Differential Integral Equations \textbf{22} (2009), no. 5-6, 447--464.
\bibitem{K11} N. Kishimoto, \textit{Local well-posedness for the Zakharov system on multidimensional torus}, preprint (2011). \texttt{arXiv:1109.3527}
\bibitem{KT10} N. Kishimoto and K. Tsugawa, \textit{Local well-posedness for quadratic nonlinear Schr\"odinger equations and the ``good'' Boussinesq equation}, Differential Integral Equations \textbf{23} (2010), no. 5-6, 463--493.
\bibitem{L93} F. Linares, \textit{Global existence of small solutions for a generalized Boussinesq equation}, J. Differential Equations \textbf{106} (1993), no. 2, 257--293.
\bibitem{M11p} L. Molinet, \textit{Sharp ill-posedness results for the KdV and mKdV equations on the torus}, preprint (2011). \texttt{arXiv:1105.3601}
\bibitem{M74} T. Muramatu, \textit{On Besov spaces and Sobolev spaces of generalized functions definded on a general region}, Publ. Res. Inst. Math. Sci. \textbf{9} (1973/74), 325--396.
\bibitem{MT04} T. Muramatu and S. Taoka, \textit{The initial value problem for the 1-D semilinear Schr\"odinger equation in Besov spaces}, J. Math. Soc. Japan \textbf{56} (2004), no. 3, 853--888.
\bibitem{OS12} S. Oh and A. Stefanov, \textit{Improved local well-posedness for the periodic ``good'' Boussinesq equation}, preprint (2012). \texttt{arXiv:1201.1942}
\bibitem{S90} R.L. Sachs, \textit{On the blow-up of certain solutions of the ``good'' Boussinesq equation}, Appl. Anal. \textbf{36} (1990), no. 3-4, 145--152.
\bibitem{Taobook} T. Tao, Nonlinear dispersive equations. Local and global analysis. CBMS Regional Conference Series in Mathematics, 106. American Mathematical Society, Providence, RI, 2006.
\bibitem{Z74} V.E. Zakharov, \textit{On stochastization of one-dimensional chains of nonlinear oscillators}, Sov. Phys. JETP \textbf{38} (1974), no. 1, 108--110.
\end{thebibliography}
\end{document}